\documentclass{amsart}
\usepackage{amssymb,latexsym}
\usepackage{amsmath}
\usepackage{graphicx}
\usepackage{amscd}
\usepackage{color}
\usepackage{enumerate}
\newenvironment{enumeratei}{\begin{enumerate}[\upshape (i)]}{\end{enumerate}}
\newcommand \url [1] {\tt{#1}}
\numberwithin{equation}{section}
\theoremstyle{plain}
 \newtheorem{theorem}{Theorem}[section]
 \newtheorem{lemma}[theorem]{Lemma}
 \newtheorem{proposition}[theorem]{Proposition}
 
 \newtheorem{corollary}[theorem]{Corollary}
\theoremstyle{definition}
 \newtheorem{definition}[theorem]{Definition}
 \newtheorem{remark}[theorem]{Remark}
 
 \newtheorem{example}[theorem]{Example}

\theoremstyle{remark}

\newcommand\dslRtBxB{\textup{(4.3)}}
\newcommand \figinsgad{8} 
\newcommand \insgadlemma{4.5}
\newcommand \cloops{\textup{(4.9)}}
\newcommand \intops{\textup{(4.10)}}
\newcommand \nRlDeF{\textup{(4.11)}}
\newcommand \insgada{\textup{(4.12)}}

\newcommand \insgadd{\textup{(4.15)}}
\newcommand \figmthreethree {9} 
\newcommand \rank [1] {r(#1)}
\newcommand \keylemma{4.6}
\newcommand \primprojlemma{4.2}
\newcommand \technicallemma{4.1}
\newcommand \nablIfffcl{\textup{(4.27)}}
\newcommand \dwhTppwhjS{\textup{(4.28)}}
\newcommand \xilemma{4.7}
\newcommand \cordiakern {\kern 0.58cm}
\newcommand \bdiakern {\kern 0.1cm}
\newcommand \balpha {\boldsymbol\alpha}
\newcommand \plusskip{\kern 5 pt}
\newcommand \minusskip{\kern -8 pt}
\newcommand \secatt {\textup{\tbf{Set}}}
\newcommand \latf {\mathbf{Lat}_{5}}
\newcommand \bposets {\mathbf{Pos}_{01}^{\tsssty +}} 
\newcommand \latasdf {\mathbf{Lat}^{\textup{sd}}_{5}}
\newcommand \cat [1]{\mathbf{#1}}
\newcommand \alllat {\mathbf{Lat}_{01}}
\newcommand \celfunct {{F_{\kern-1.5pt  \tsssty{com}}}}
\newcommand \ccelfunct [1] {{F_{\kern-1.5pt  \tsssty{com}}^{#1} }}
\newcommand \posetfunct {{F_{\kern-1.5pt \tsssty{pos}}}}
\newcommand \psf [1] {{\overline #1}}
\newcommand \psv [1] {{\overline #1}{}'}
\newcommand \psnv [1] {{\overline #1}\kern 1pt'}
\newcommand \liftfunct {{E_{\tsssty{Lift}}}}
\newcommand \fofunct {{G_{\kern-1.5pt \tsssty{forg}}}}
\newcommand \compfunct {{G_{\kern-1.5pt \tsssty{prod}}}}
\newcommand \foifunct {{G_{\kern-1.5pt \tsssty{forg}}^{-1}}}
\newcommand \celtransf {{{\boldsymbol\pi}^{\tsssty{com}}}}
\newcommand \catceltransf [1] {{\boldsymbol\pi}^{\tsssty{com},#1}}
\newcommand \xceltransf [1] {{\boldsymbol\pi}_{\kern-0.5pt #1}^{\tsssty{com}}}
\newcommand \fthird {\textup{pr}^{\tsssty{(3)}}}
\newcommand \third [1] {\fthird(#1)}
\newcommand \xpceltransf [2] {{{\boldsymbol\pi}^{\tsssty{com}}_{\kern-0.5pt #1}(#2)}}
\newcommand \tautr{\boldsymbol{\kappa}}
\newcommand \xtautr [1] {\boldsymbol{\kappa}_{#1}}
\newcommand \trivcom {\vec v^{\kern 1pt\textup{triv}}}
\newcommand \obj [1] {\textup{Ob}(#1)}
\newcommand \fmor  {\textup{Mor}}
\newcommand \mor [1] {\fmor(#1)}
\newcommand \idfunct[1] {I_{#1}}
\newcommand \eidfunct[2] {I_{#1,#2}}
\newcommand \tsssty[1]{{\scriptscriptstyle{\textup{#1}}}}
\newcommand \zetaf [3] {\zeta_{#1,#2,#3}}
\newcommand \pairs [1] {{\textup{Pairs}^{\leq}(#1)}}
\newcommand \covpairs [1] {{\textup{Pairs}^{\prec}(#1)}}

\newcommand\gad {{G}}
\newcommand\agad{{\alg\gad}}

\newcommand \pupgad [1]{G^{\sst{\upsign}}_{#1}}
\newcommand \apupgad [1]{\alg G^{\sst{\upsign}}_{#1}}
\newcommand \pdngad [1]{G^{\sst{\dnsign}}_{#1}}
\newcommand \pqmgad [1]{G^{\sst{\qmsign}}_{#1}}
\newcommand \pstrgad [2]{G^{{\textup{#1}}}_{#2}}
\newcommand \apdngad [1]{\alg G^{\sst{\dnsign}}_{#1}}
\newcommand \cpq[1]{c_{#1}^{p q}}
\newcommand \dpq[1]{d_{#1}^{p q}}
\newcommand \epq{e^{p q}}
\newcommand \dcpq[1]{c^{#1}_{p q}} 
\newcommand \ddpq[1]{d^{#1}_{p q}}
\newcommand \depq{e_{p q}}
\newcommand \adbgad [1]{\alg G^{\sst{\dblesign}}_{#1}}
\newcommand \dbgad [1]{G^{\sst{\dblesign}}_{#1}}
\newcommand \pupmap [2]{g^{\sst{\upsign}}_{#1#2}}
\newcommand \pdnmap [2]{g^{\sst{\dnsign}}_{#1#2}}
\newcommand\upsign{\textup{up}}
\newcommand\dnsign{\textup{dn}}
\newcommand\dblesign{\textup{db}}
\newcommand\qmsign{\forall}

\renewcommand\rho{\varrho}
\newcommand \trup{{\scriptscriptstyle{{{\mathord{\pmb{\vartriangle}  }}}}}}
\newcommand \uins [1] {{#1}^\trup}
\newcommand \upins [4] {{#1}^\trup_{#2#3#4}}
\newcommand \lequins{\mathrel{\uins{\mathord{\leq}}}}
\newcommand \veeuins{\mathrel{\uins{\mathord{\vee}}}}
\newcommand \wedgeuins{\mathrel{\uins{\mathord{\wedge}}}}

\newcommand \Mnh {M_{4\times3}}
\newcommand\length{\textup{length}}
\newcommand \inter [1]{{#1}^{\mathord{-}ZU}}

\newcommand\dn {^{\sst{\dnsign}}}
\newcommand\up {^{\sst{\upsign}}}
\newcommand \Nla {L}

\newcommand \fcs[1]{{{#1}^\ast}}
\newcommand \acs[1]{{{#1}_\ast}}
\newcommand \fhat[1]{{{#1}^\bullet}}
\newcommand \ahat [1]{{{#1}_\bullet}}
\newcommand \leqnu {\mathrel{\leq_\nu}}
\newcommand \sst [1] {\scriptscriptstyle #1}
\newcommand \fprinc {\textup{Princ}}
\newcommand \princ[1] {\fprinc(#1)}
\newcommand \cg[2] {\textup{con}(#1,#2)}
\newcommand \congen[2] {\textup{con}_{#1}(#2)}
\newcommand \cgi[3] {\textup{con}_{#1}(#2,#3)}
\newcommand\alg [1] {{\mathcal #1}}
\newcommand \quo[2] {\textup{quo}(#1,#2)}
\newcommand \quos[1] {\textup{quo}(#1)}
\newcommand \bquos[1] {\textup{quo}\bigl(#1\bigr)}
\newcommand \Quos[1] {\textup{quo}\Bigl(#1\Bigr)}
\newcommand \iquos[2] {\textup{quo}_{#1}(#2)}
\newcommand \blokk[2] {#1/#2}
\newcommand \restrict[2] {{#1\rceil_{#2}}}
\newcommand\iideal[2]{\mathord\downarrow_{\kern-2pt #1\kern 1pt} #2}
\newcommand\ifilter[2]{\mathord\uparrow_{\kern-2pt #1\kern 1pt} #2}
\newcommand \tuple [1] {\langle #1\rangle}
\newcommand \pair [2] {\tuple{#1,#2}}
\renewcommand\emptyset{\varnothing}
\newcommand\red[1]{{\textcolor{red}{#1}}}
\newcommand \tbf [1] {\textbf{#1}} 
\newcommand \set[1] {\{#1\}}

\newcommand \then {\mathrel{\Rightarrow}} 
\newcommand \nablaell [1] {\nabla_{\kern -2pt #1}}
\newcommand\init [1] {#1} 

\begin{document}
\title[Cometic functors]
{Cometic functors for small concrete categories and an application}

\author[G.\ Cz\'edli]{G\'abor Cz\'edli}
\email{czedli@math.u-szeged.hu}
\urladdr{http://www.math.u-szeged.hu/\textasciitilde{}czedli/}
\address{University of Szeged, Bolyai Institute, 
Szeged, Aradi v\'ertan\'uk tere 1, HUNGARY 6720}

\thanks{This research was supported by
NFSR of Hungary (OTKA), grant number K 115518}

\subjclass[2000] {18B05%
\red{.\hfill August 4, 2015}}


\begin{abstract} Our goal is to derive some families of maps, also known as functions, from injective maps and surjective maps; this can be useful in various fields of mathematics. Let $\cat A$ be a small concrete category. We
define a functor $\celfunct$, called \emph{cometic functor}, from $\cat A$ to the category $\secatt$
and a natural transformation $\celtransf$, called \emph{cometic projection}, from $\celfunct$ to the inclusion functor of $\cat A$ into $\secatt$ such that  the $\celfunct$-image of every monomorphism of $\cat A$ is an  injective map and the components of  $\celtransf$ are  surjective maps. Also, we give a nontrivial application of $\celfunct$ and $\celtransf$.
\end{abstract}

\maketitle
\section{Prerequisites and outline}
This paper consists of an easy category theoretical part followed by a more involved lattice theoretical part.  

The \emph{category theoretical} first part, which consists of Sections~\ref{sectioncatintro} and \ref{secmorecelestial}, is devoted to certain families of \emph{maps}, also known as \emph{functions}.  
Only few concepts are  needed from category theory; all of them are easy and their definitions will be recalled in the paper. Hence, there is no prerequisite for this part. Our purpose is to derive some families of maps from injective maps and surjective maps. This part can be interesting in various fields of algebra and even outside algebra.

The \emph{lattice theoretical} second part is built on the first part.  
The readers of the second part  are not assumed to have deep knowledge of lattice theory; a little part of any book on lattices, including \init{G.\ }Gr\"atzer~\cite{ggglt} and \init{J.\,B.\ }Nation~\cite{nation},  is sufficient.  

The paper is structured as follows. In Section~\ref{sectioncatintro}, we recall some basic concepts from category theory.  In Section~\ref{secmorecelestial}, we 
introduce cometic functors and cometic projections, and  prove  Theorem~\ref{thmcat} on them. 
In Section~\ref{seclatintro}, we formulate  Theorem~\ref{thmlat} on the representation of families of monotone maps by principal lattice congruences. 
Section~\ref{secquasicolor} tailors the toolkit developed for quasi-colored lattices in \init{G.~}Cz\'edli~\cite{czginjlatcat} to the present environment; when reading this section, \cite{czginjlatcat} should be nearby. In Section~\ref{secfromquasitohomo}, we prove a lemma that allows us to work with certain homomorphisms efficiently.
Finally, with the help of cometic functors and projections, Section~\ref{sectlatcompl} completes the paper by proving Theorem~\ref{thmlat}.

\section{Introduction to the category theory part}\label{sectioncatintro}
\subsection{Notation, terminology, and the rudiments} 
Recall that a \emph{category} $\cat A$ is a system $\tuple{\obj{\cat A}, \mor{\cat A},\circ}$ formed from a class $\obj{\cat A}$ of \emph{objects}, a class $\mor{\cat A}$ of \emph{morphisms}, and a partially defined binary operation $\circ$ on $\mor{\cat A}$ such that $\cat A$ satisfies certain axioms.  
Each $f\in \mor{\cat A}$ has a \emph{source object} $X\in \obj{\cat A}$ and a \emph{target object} $Y\in \obj{\cat A}$; the collection of morphisms with source object $X$ and target object $Y$ is denoted by $\fmor(X,Y)$ or $\fmor_{\cat A}(X,Y)$. 
The axioms require that $\fmor(X,Y)$ is a set for all $X,Y\in \obj{\cat A}$, every $\fmor(X,X)$
contains a unique \emph{identity morphism} $1_X$,
 $f\circ g$ is defined and belongs to $\fmor(X,Z)$ iff $f\in \fmor(Y,Z)$ and $g\in \fmor(X,Y)$, this multiplication is associative, and the identity morphisms are left and right units with respect to the multiplication. Note that $\fmor(X,Y)$ is often called a \emph{hom-set} of $\cat A$ and $\mor{\cat A}$ is the disjoint union of the hom-sets of $\cat A$.
If $\cat A$ and $\cat B$ are categories such that $\obj{\cat A}\subseteq \obj{\cat B}$ and 
$\mor{\cat A}\subseteq \mor{\cat B}$, then $\cat A$ is a \emph{subcategory} of $\cat B$. 
If  $\cat A$  is a category and $\obj{\cat A}$ is a set, then $\cat A$ is said to be a \emph{small category}. 

\begin{definition}\label{defconcrcat}
If $\cat A$ is a category such that 
\begin{enumeratei}
\item  every object of  $\cat A$ is a set,
\item  for all $X,Y\in\obj{\cat A}$
and $f\in\fmor(X,Y)$, $f$ is a map from $X$ to $Y$,  and 
\item the operation is the usual composition of maps,
\end{enumeratei}
then $\cat A$ is a \emph{concrete category}. Note the rule $(f\circ g)(x)=f\bigl(g(x)\bigr)$, that is, we compose maps from right to left. Note also that $\fmor(X,Y)$ does not have to contain all possible maps from $X$ to $Y$. The category of all sets with all maps between sets will be denoted by $\secatt$.  
\end{definition}

\begin{remark} 
In category theory, the concept of 
concrete categories is usually  based on forgetful functors and it has a more general meaning. Since this paper is not only for category theorists, we adopt Definition~\ref{defconcrcat}, which is conceptually simpler but, apart from mathematically insignificant technicalities, will not reduce the generality of our result, Theorem~\ref{thmcat}.
\end{remark}

For an arbitrary category $\cat A$ and $f\in \mor{\cat A}$, if $f\circ g_1= f\circ g_2$ implies $g_1=g_2$ for all $g_1,g_2\in \mor{\cat A}$ such that both  $f\circ g_1$ and $f\circ g_2$ are defined, then $f$ is a \emph{monomorphism} in $\cat A$. 
Note that if $\cat A$ is a subcategory of $\cat B$, then a monomorphism of $\cat A$ need not be a monomorphism of $\cat B$. 
In a concrete category, an injective morphism is always a monomorphism but not conversely. The opposite (that is, left-right dual) of the concept of monomorphisms is that of \emph{epimorphisms}. An \emph{isomorphism} in $\cat A$ is a morphism that is both mono and epi. Next, let $\cat A$ and $\cat B$ be categories. 
An assignment $F\colon \cat A\to \cat B$ is a \emph{functor} if $F(X)\in \obj{\cat B}$ for every $X\in\obj{\cat A}$, $F(f)\in\fmor_{\cat B}(F(X),F(Y))$ for every $f\in\fmor_{\cat A}(X,Y)$, $F$ commutes with $\circ$, and $F$ maps the identity morphisms to identity morphisms. 
If $F(f)=F(g)$ implies $f=g$ for all $X,Y\in \obj{\cat A}$ and all $f,g\in\fmor_{\cat A}(X,Y)$, then $F$ is called a \emph{faithful functor}. Although category theory seems to avoid talking about equality of objects, to make our theorems stronger, we introduce the following concept.

\begin{definition}
For categories $\cat A$ and $\cat B$ and a functor $F\colon \cat A\to \cat B$, $F$ is a \emph{totally faithful functor} if, for all $f,g\in \mor{\cat A}$, $f\neq g$ implies that $F(f)\neq F(g)$. 
\end{definition}

\begin{remark} Let  $F\colon \cat A\to \cat B$ be a functor. Then $F$ is  totally faithful iff it is  faithful and, for all $X,Y\in\obj{\cat A}$, $X\neq Y$ implies $F(X)\neq F(Y)$.
\end{remark}

\begin{proof} Assume that $F$ is totally faithful, and let $X,Y\in \obj{\cat A}$ such that $X\neq Y$. Then
$1_X\neq 1_Y$, so $1_{F(X)}= F(1_X)\neq F(1_Y) = 1_{F(Y)}$, and we conclude that $F(X)\neq F(Y)$. To see the converse implication, let $f_1\in \fmor_{\cat A}(X_1,Y_1)$ and $f_2\in \fmor_{\cat A}(X_2,Y_2)$ such that $f_1\neq f_2$.  If $\pair{X_1}{Y_1}=\pair{X_2}{Y_2}$, then $F(f_1)\neq F(f_2)$ since $F$ is faithful. 
Otherwise, $\pair{F(X_1)}{F(Y_1)} \neq \pair{F(X_2)}{F(Y_2)}$ by the assumption, and $f_1\neq f_2$ follows from $\fmor_{\cat B}(F(X_1),F(Y_1))
\cap \fmor_{\cat B}(F(X_2),F(Y_2))=\emptyset$, since $\mor{\cat B}$ is the \emph{disjoint} union of the hom-sets of $\cat B$.  
\end{proof}

If $\cat A$ is a subcategory of $\cat B$, then the 
\begin{equation}
\text{\emph{inclusion functor} }\eidfunct{\cat A}{\cat B}\colon \cat A\to \cat B\text{ is defined}
\label{eqrinclusionfunctor}
\end{equation}
by the rules $\eidfunct{\cat A}{\cat B}(X)=X$ for $X\in \obj{\cat A}$ and
$\eidfunct{\cat A}{\cat B}(f)=f$ for $f\in \mor{\cat A}$. The  \emph{identity functor} 
$\idfunct{\cat A}\colon \cat A\to \cat A$ is the particular case $\cat 
B=\cat A$, that is, $\idfunct{\cat A}:=\eidfunct{\cat A}{\cat A}$.
For a functor $F\colon\cat A\to \cat B$, the \emph{$F$-image} of $\cat A$ is the category 
\begin{equation*}
F(\cat A)=\tuple{\set{F(X):X\in\obj{\cat A}},\, \set{F(f):f\in\mor{\cat A}}, \circ}\text.
\end{equation*}

Next, let $F$ and $G$ be functors from a category $\cat A$ to a category $\cat B$. 
A \emph{natural transformation} $\tautr\colon F\to G$ is a system $\tuple{\tautr_X: X\in\obj{\cat A}}$ of morphisms of $\cat B$ such that the \emph{component} $\tautr_X$ of $\tautr$ at $X$ belongs to $\fmor_{\cat B}(F(X),G(X))$ for every $X\in\obj{\cat A}$, and for every $X,Y\in\obj{\cat A}$ and every $f\in \fmor_{\cat A}(X,Y)$, the diagram 
\begin{equation*}
\begin{CD}
F(X)  @>{F(f)}>> F(Y) \\
@V{\tautr_X}VV   @V{\tautr_Y}VV   \\
G(X) @>{G(f)}>>  G(Y)
\end{CD}
\end{equation*}
commutes, that is, 
$\tautr_Y\circ F(f)= G(f)\circ \tautr_X$.
If all the components $\tautr_X$ of $\tautr$ are isomorphisms in $\cat B$, then $\tautr$ is a \emph{natural isomorphism}. 
If there is a natural isomorphism $\tautr\colon F\to G$, then $F$ and $G$ are \emph{naturally isomorphic functors}. 
Note that naturally isomorphic functors are, sometimes, also called \emph{naturally equivalent}.

\section{Cometic functors and projections}\label{secmorecelestial}
Our purpose is to derive some families of maps from injective and surjective maps. To do so, we introduce some concepts. The third 
component of an arbitrary triplet  $\tuple{x,y,z}$ is obtained by the \emph{third projection} $\fthird$, in notation, 
\[\third{\tuple{x,y,z}}=z\text.
\]
\begin{definition}\label{defcomet}
Given a small concrete category $\cat A$, a triplet $c=\tuple{f,x,y}$ is an \emph{eligible triplet} of $\cat A$ if there exist $X,Y\in\obj{\cat A}$ such that  $f\in \fmor_{\cat A}(X,Y)$, $x\in X$, $y\in Y$,  and $f(x)=y$. The third component  of  $c=\tuple{f,x,y}$ will also be denoted by  
\[\xpceltransf Y{\tuple{f,x,y}} := \third{\tuple{f,x,y}}
 y=f(x),\,\text{ provided that}\, y\in Y\text. 
\]
For $x\in X\in \obj{\cat A}$, 
\begin{equation*}\trivcom(x)=\trivcom_X(x)\,\text{ denotes }\, \tuple{1_X,x,x},
\end{equation*}
the \emph{trivial triplet} at $x$. Note the obvious rule
\begin{equation}
\xpceltransf X{\trivcom_X(x)}=x,\,\text{ for }\,x\in X\text.
\label{eqrtrrtrainv} 
\end{equation}
\end{definition}

\begin{definition}\label{defcomeTic}
Given a small concrete category $\cat A$ (see Definition~\ref{defconcrcat}), we define the \emph{cometic functor} 
\[\celfunct=\ccelfunct{\cat A}\colon \cat A\to \secatt 
\]
associated with $\cat A$ as follows. 
For each $Y\in \obj{\cat A}$, we let 
\[\celfunct(Y):=\set{\tuple{f,x,y}: \tuple{f,x,y}\text{ is an eligible triplet of }\cat A\text{ and }y\in Y}\text.
\]
For $X,Y\in\obj{\cat A}$ and $g\in\fmor_{\cat A}(X,Y)$, we define $\celfunct(g)$ as the map
\begin{equation*}
\begin{aligned}
\celfunct(g)\colon \celfunct(X) &\to \celfunct(Y),
\text{ defined by }\cr 
\tuple{f,x,y}&\mapsto\tuple{g\circ f, x, g(y)}\text.
\end{aligned}
\end{equation*}
Finally, the map $X\to \celfunct(X)$, defined by $x\mapsto \trivcom(x)$, will be denoted by $\trivcom_X$.
\end{definition}
We could also denote an eligible triplet $\tuple{f,x,y}$ by  $x\overset f\mapsto y$, but technically the triplet is a more convenient notation than the $f$-labeled 
``\texttt{$\backslash$mapsto}'' arrow. However, in this paragraph, let us think of  eligible triplets as arrows. The trivial arrows $\trivcom_X (x)$ with $x\in X$ 
correspond to the elements of $X$. Besides these arrows, $\celfunct(X)$ can contain many other  arrows, which are of different lengths and of different directions in space but with third components in $X$. This geometric interpretation of $\celfunct(X)$ resembles a real comet; the trivial  arrows form the nucleus while the rest of arrows the coma and the tail. This explains the adjective ``cometic''.

\begin{lemma}\label{lemmaCometic} 
$\celfunct=\ccelfunct{\cat A}$  from Definition~\ref{defcomeTic} is a totally faithful functor. 
\end{lemma}

\begin{proof} First, we prove that $\celfunct:=\ccelfunct{\cat A}$ is a functor.
Obviously, the $\celfunct$-image of an identity morphism is an identity morphism. Assume that $X,Y,Z\in \obj{\cat A}$, $f\in\fmor_{\cat A}(X,Y)$, $g\in\fmor_{\cat A}(Y,Z)$, $c=\tuple{h,x,y}\in \celfunct(X)$,  and let us compute:
\[
\begin{aligned}
\bigl(\celfunct(g)&\circ \celfunct(f)\bigr)(c)=
\celfunct(g)\bigl(\celfunct(f)(c)\bigr)\cr
&=\celfunct(g)\bigl(\tuple{f\circ h, x, f(y)}\bigr) =\tuple{g\circ (f\circ h), x, g(f(y))} \cr
&=\tuple{(g\circ f)\circ h, x, (g\circ f)(y)} =\celfunct(g\circ f)(c)\text.
\end{aligned}
\]
Hence, $\celfunct(g)\circ \celfunct(f)=\celfunct(g\circ f)$ and $\celfunct$ is a functor.
To prove that $\celfunct$ is faithful, assume that $X,Y\in\obj{\cat A}$, $f,g\in\fmor_{\cat A}(X,Y)$, and $\celfunct(f)=\celfunct(g)$; we have to show that $f=g$. This is clear if $X=\emptyset$. Otherwise, for $x\in X$,
\[
\tuple{f\circ 1_X, x,f(x)}=
\celfunct(f)(\trivcom(x))=\celfunct(g)(\trivcom(x))
=\tuple{g\circ 1_X, x,g(x)}\text.
\]
Comparing either the third components (for all $x\in X$), or the first components, we conclude that $f=g$. Thus, $\celfunct$ is faithful. Finally, if $X,Y\in\obj{\cat A}$ and $X\nsubseteq Y$, then there is an $x\in X\setminus Y$. Since $\trivcom(x)\in \celfunct(X)\setminus\celfunct(Y)$, we conclude that $\celfunct$ is totally faithful.
\end{proof}

\begin{definition} Let $\cat A$ be a small concrete category, let $ \eidfunct{\cat A}{\secatt}$ be the inclusion functor from $\cat A$ into $\secatt$, see  \eqref{eqrinclusionfunctor}, and keep Definition~\ref{defcomeTic} in mind.  Then the  transformation 
\[
\celtransf=\catceltransf{\cat A}\colon \celfunct \to \eidfunct{\cat A}{\secatt}
\]
whose components are defined  by
\[
\xceltransf X \colon \celfunct(X)\to X\quad\text{ and }\quad \xpceltransf X c:=\third c,
\]
for $X\in \obj{\cat A}$ and $c\in\celfunct(X)$, 
is the \emph{cometic projection}  associated with $\cat A$. (Note that $\xceltransf X$ is simply the restriction of the third projection $\fthird$ to $\celfunct(X)$.)
\end{definition}

\begin{lemma}\label{lemmapiisnattransf}
The cometic projection defined above is a natural transformation and its components are surjective maps. 
\end{lemma}

\begin{proof} Let $X,Y\in\cat A$ and $f\in\fmor(X,Y)$. 
We have to prove that the diagram 
\begin{equation}
\begin{CD}
\celfunct(X)  @>{\celfunct(f)}>> \celfunct(Y) \\
@V{\xceltransf X}VV   @V{\xceltransf Y}VV   \\
X @>{f}>>  Y
\end{CD}
\label{CDcelUUtdef}
\end{equation}
commutes. For an arbitrary triplet $c=\tuple{h,x,y}\in \celfunct(X)$, 
\allowdisplaybreaks{
\begin{align*}
\bigl(\xceltransf Y\circ\celfunct(f)\bigr)(c) &= \xceltransf Y \bigl(\celfunct(f)(c)\bigr)  = \xceltransf Y \bigl(\tuple{f\circ h,x,f(y)}\bigr)\cr 
& = f(y)=   f\bigl(\xceltransf X(c)\bigr)
= (f\circ \xceltransf X)(c),
\end{align*}}%
which proves the commutativity of  \eqref{CDcelUUtdef}. Finally, for $X\in \obj{\cat A}$ and $x\in X$, $x= \xceltransf X(\trivcom (x))$.  Thus, the components of $\celtransf$ are surjective. 
\end{proof}

Now, we are in the position to state the main result of this section; it also summarizes  Lemmas~\ref{lemmaCometic} and \ref{lemmapiisnattransf}.

\begin{theorem}\label{thmcat}
Let a  $\cat A$ be a small concrete category.
\begin{enumerate}[\kern5pt \normalfont(A)]
\item\label{thmcatA} For the
cometic functor $\celfunct=\ccelfunct{\cat A}$ and the cometic projection $\celtransf=\catceltransf{\cat A}$ associated with $\cat A$, the following hold.
\begin{enumeratei} 
\item\label{thmcatAa} $\celfunct\colon \cat A\to \secatt$ is a totally faithful functor, and  $\celtransf\colon \celfunct\to \eidfunct{\cat A}\secatt$ is a natural transformation whose components are surjective maps.   
\item\label{thmcatAb} For every $f\in\mor{\cat A}$, $f$ is a monomorphism in $\cat A$ if and only if $\celfunct(f)$ is an injective map. 
\end{enumeratei}
\item\label{thmcatB} Whenever $F\colon\cat A\to\secatt$ is a  functor and $\tautr\colon F\to \eidfunct{\cat A}\secatt$ is a natural transformation whose components are surjective maps, then for every morphism $f\in \mor{\cat A}$, if  $F(f)$ is an injective map, then $f$ is a monomorphism in $\cat A$.
\end{enumerate}
\end{theorem}

By part \eqref{thmcatB}, we cannot ``translate'' more morphisms to injective maps than those translated by $\celfunct$. In this sense, part \eqref{thmcatB} is the converse of part \eqref{thmcatA} (with less assumptions on the functor).

\begin{proof}[Proof of Theorem~\ref{thmcat}]  \eqref{thmcatAa} is the conjunction of Lemmas~\ref{lemmaCometic} and \ref{lemmapiisnattransf}.

To prove part  \eqref{thmcatB}, let $\cat A$ be a small concrete category, let $F\colon \cat A\to \secatt$ be a functor, and let $\tautr\colon F\to \eidfunct{\cat A}\secatt$ be a natural transformation with surjective components. Assume that $Y,Z\in \obj{\cat A}$ and $f\in \fmor_{\cat A}(Y,Z)$ such that $F(f)$ is injective. To prove that $f$ is a monomorphism in $\cat A$, let $X\in \obj{\cat A}$ and $g_1,g_2\in \fmor_{\cat A}(X,Y)$ such that $f\circ g_1=f\circ g_2$; we have  to show that $g_1=g_2$. 
That is, we have to show that, for an arbitrary $x\in X$, $g_i(x)$ does not depend on $i\in\set{1,2}$.
By the surjectivity of $\tautr_X$,  we can pick an element $a\in \celfunct(X)$ such that $x=\tautr_X(a)$. 
Since $f\circ g_1=f\circ g_2$, 
\[
F(f)\bigl( F(g_i)(a)\bigr)=
\bigl(F(f)\circ F(g_i)\bigr)(a)= 
F(f\circ g_i)(a)
\]
does not depend on $i\in\set{1,2}$. Hence, the injectivity of  $F(f)$ yields that $F(g_i)(a)$ does not depend on $i\in\set{1,2}$. Since $\tautr$ is a natural transformation, 
\begin{equation*}
\begin{CD}
F(X)  @>{F(g_i)}>> F(Y) \\
@V{\tautr_X}VV   @V{\tautr_Y}VV   \\
X @>{g_i}>>  Y
\end{CD}
\end{equation*}
is a commutative diagram, and we obtain that
\begin{align*}
g_i(x)=g_i\bigl(\tautr_X(a)\bigr)= 
(g_i\circ\tautr_X)(a)= \bigl(\tautr_Y\circ F(g_i)\bigr)(a) = \tautr_Y  \bigl(F(g_i)(a)\bigr)\text.
\end{align*}
Hence, $g_i(x)$ does not depend on $i\in\set{1,2}$, because neither does $F(g_i)(a)$. Consequently,  $g_1=g_2$. Thus, $f$ is a monomorphism, proving  part  \eqref{thmcatB}.

To prove  the ``only if'' direction of  \eqref{thmcatAb}, assume that $X,Y\in \obj{\cat Y}$ and $f\in\fmor_{\cat A}(X,Y)$ such that $f$ is a monomorphism in $\cat A$. We have to show that $\celfunct(f)$ is injective. 
To do so, let $c_i=\tuple{h_i,x_i,y_i}\in \celfunct(X)$ such that $\celfunct(f)(c_1)=\celfunct(f)(c_2)$. Since the middle components of 
\[\tuple{f\circ h_1,x_1,f(y_1)} = \celfunct(f)(c_1)=\celfunct(f)(c_2)= \tuple{f\circ h_2,x_2,f(y_2)}\text.
\]
are equal, we have that $x_1=x_2$. Since $f$ is a  monomorphism, the equality of the first components yields that $h_1=h_2$. Since $c_1$ and $c_2$ are eligible triplets, the first two components determine the third. Hence, $c_1=c_2$  and $\celfunct(f)$ is injective, as required. This proves the ``only if'' direction  of part \eqref{thmcatAb}.

Finally, the  ``if'' direction of  \eqref{thmcatAb} follows from
\eqref{thmcatAa} and \eqref{thmcatB}.
\end{proof}

\begin{remark}
There are many examples of monomorphisms in small concrete categories that are not injective. For example, let $f\colon X\to Y$ be a non-injective map between two distinct sets. Consider the category $\cat A$ with $\obj{\cat A}=\set{X,Y}$ and $\mor{\cat A}=\set{1_X, 1_Y,f}$; then $f$ is a monomorphism in $\cat A$. For a bit more general example, see Example~\ref{exAmplesubcatE}.
\end{remark}

\begin{remark} Let $\cat A$ be as in Theorem~\ref{thmcat}, $X,Y\in \obj{\cat A}$, and $f\in \fmor(X,Y)$. Since $\trivcom_X$ from Definition~\ref{defcomeTic} is a right inverse of $\xceltransf X$, the commutativity of \eqref {CDcelUUtdef}
yields easily that 
$f=\xceltransf Y \circ \celfunct(f)\circ \trivcom_X$.  Note, however, that $\trivcom$ is not a natural transformation in general.
\end{remark}

\begin{remark} Let $\cat A$ be as in Theorem~\ref{thmcat}. As an easy consequence of the theorem, every monomorphism of $\celfunct(\cat A)$ is an injective map. In this sense, 
the category $\celfunct(\cat A)$ is ``better'' than $\cat A$. Since $\celfunct(\cat A)$ is obtained by the cometic functor, one might, perhaps, call it the \emph{celestial category} associated with $\cat A$. 
\end{remark}

\section{Introduction to the lattice theory part}\label{seclatintro}  
From now on, the paper is mainly for lattice theorists. However, the reader is not assumed to have deep knowledge of lattice theory; a little part of any book on lattices, including \init{G.\ }Gr\"atzer~\cite{ggglt} and \init{J.\,B.\ }Nation~\cite{nation},  is sufficient.

Motivated by the history of the congruence lattice representation problem, which  culminated in \init{F.\ }Wehrung~\cite{wehrung} and  \init{P.\ }R\r{u}\v{z}i\v{c}ka~\cite{ruzicka},
\init{G.\ }Gr\"atzer in  \cite{ggprincl}  has recently started an analogous new topic of lattice theory.  Namely, for a lattice $L$, let $\princ L=\pair{\princ L}{\subseteq}$ denote the   ordered set  of principal congruences of $L$. A congruence is \emph{principal} if it is generated by a pair $\pair ab$ of elements.
Ordered sets (also called  partially ordered sets or posets) and lattices with 0 and 1 are called \emph{bounded}. If $L$ is a bounded lattice, then $\princ L$ is a bounded ordered set.  Conversely,  \init{G.\ }Gr\"atzer~\cite{ggprincl} proved that  every  bounded  ordered set $P$ is isomorphic to $\princ L$ for an appropriate bounded lattice $L$ of length 5.  The ordered sets $\princ L$ of countable but not necessarily bounded lattices $L$ were characterized in \init{G.\ }Cz\'edli~\cite{czgprincc}. 

\begin{definition}
We define the following four categories.
\begin{enumeratei}
\item $\alllat$ is the category of at least 2-element bounded lattices with $\set{0,1}$-preserving lattice homomorphisms.
\item  $\latf$ is the category of lattices of length 5 with $\set{0,1}$-preserving lattice homomorphisms.
\item $\latasdf$ is the category of selfdual bounded lattices of length 5 with $\set{0,1}$-preserving lattice homomorphisms.
\item $\bposets$ is the category of at least 2-element bounded ordered sets with 
$\set{0,1}$-preserving monotone maps.
\end{enumeratei}

\end{definition}

Note that $\latasdf$ is a subcategory of $\latf$, which is a subcategory of $\alllat$. Note also that if  $X$ and $Y$ are ordered sets and $|Y|=1$, then $\fmor(X,Y)$ consists of the trivial map and $\fmor (Y,X)\neq \emptyset$ iff $|X|=1$.  Hence, we do not loose anything interesting by excluding the  singleton ordered sets from $\bposets$. A similar comment applies for singleton lattices, which are excluded from $\alllat$.

For an algebra  $A$ and $x,y\in A$, the principal congruence generated by $\pair x y$ is denoted by $\cg x y$ or $\cgi A x y$. 
For lattices, the following observation is due to \init{G.\ }Gr\"atzer~\cite{gghomoprinc}; see also 
\init{G.\ }Cz\'edli~\cite{czgsingleinjectiveprinc} for the injective case. Note that $\princ A$ is meaningful for every algebra $A$. 

\begin{remark} \label{remarkzetafAB}
If $A$ and $B$ are algebras of the same type and   $f\colon A\to B$ is a homomorphism, then \begin{equation}
\begin{aligned}
\zetaf f{A}{B}\colon\princ {A}&\to \princ{B} 
\text{, defined by }\cr
\cgi {A}x y&\mapsto \cgi {B}{f(x)}{f(y)},
\end{aligned}
\label{eqrefsljGfB}
\end{equation}
is a $0$-preserving monotone map. Consequently, for every concrete category $\cat A$ of similar algebras with all homomorphisms as morphisms, $\fprinc$ is a functor from $\cat A$ to the category of ordered sets having 0 with 0-preserving monotone maps. 
\end{remark}

\begin{proof} We only have to prove that $\zetaf fAB$ is a well-defined map, since the rest of the statement is obvious. That is, we have to prove that if $\cgi Aab=\cgi Acd$, then $\cgi B{f(a)}{f(b)}=\cgi B{f(c)}{f(d)}$. Clearly,  it suffices to prove that if $a,b,c,d\in A$ such that $\pair ab \in\cgi {A}c d$, then 
$\pair {f(a)}{f(b)} \in\cgi {B}{f(c)}{f(d)}$.
According to a classical lemma of \init{A.~I.~}Mal'cev~\cite{malcevlemma}, see also \init{E.~}Fried, \init{G.~}Gr\"atzer and \init{R.~}Quackenbush~\cite[Lemma 2.1]{friedggq}, the containment $\pair ab \in\cgi {A}c d$ is witnessed by a system of certain equalities among terms applied for certain elements of $A$. Since $f$ preserves these equalities, $\pair {f(a)}{f(b)} \in\cgi {B}{f(c)}{f(d)}$, as required.
\end{proof}

It follows from Remark~\ref{remarkzetafAB} that
\begin{equation}
\begin{aligned}
\fprinc\colon \latasdf &\to \bposets, \text{defined by }\cr 
X&\mapsto \princ X\text{ for }X\in \obj{\latasdf}\text{ and }\cr
f &\mapsto \zetaf fXY(f)\text{ for }f\in\fmor(X,Y),
\end{aligned}
\label{eqrfuncPrincDf}
\end{equation}
is a functor. 
Note that $\fprinc$ could similarly be defined with $\alllat$ or  $\latf$ as its domain category. 
Prior to Definition~\ref{defgilwehr}, it is reasonable to remark the following.

\begin{remark}\label{remarkepisurjmonoinj}
In $\bposets$, the monomorphisms and the epimorphisms are exactly the injective maps and the surjective maps, respectively. Hence, the isomorphisms in category theoretical sense are  precisely  the isomorphisms in order theoretical sense. 
\end{remark}

\begin{proof} It is well-known that an injective map is a monomorphism and a surjective map is an epimorphism. To prove the converse, assume that $f\colon X\to Y$ is a non-injective morphism in $\bposets$. Pick $x_1\neq x_2\in X$ such that $f(x_1)= f(x_2)$, and let $Z=\set{0\prec z\prec 1}$ be the three-element chain. Define the $\set{0,1}$-preserving monotone map $g_i\colon Z\to X$ by the rule $g_i(z)=x_i$. Since $g_1\neq g_2$ but  $f\circ g_1=f\circ g_2$, $f$ is not injective. 
Next, assume that $f\colon X\to Y$ is a non-surjective morphism of $\bposets$,  pick an $y\in Y\setminus f(X)$, and pick two elements,  $y_1$ and $y_2$, outside $Y$. On the set $Y':=(Y\setminus\set{y})\cup\set{y_1,y_2}$, define the ordering relation by the rule $u < v$ iff either $\set{u,v}\cap\set{y_1,y_2}=\emptyset$ and $u<_Y v$, or $u=y_i$ and $y<_Y v$, or $v=y_i$ and $u<_Y y$ for some $i\in\set{1,2}$. Note that $y_1$ and $y_2$ are incomparable, in notation, $y_1\parallel y_2$ in $Y'$. Let $g_i\colon Y\to Y'$ be defined by $u\mapsto u$ if $u\neq y$ and $y\mapsto y_i$. Then $g_1,g_2\in\mor{\bposets}$, $g_1\circ f=g_2\circ f$ but $g_1\neq g_2$, showing that $f$ is not an epimorphism.  
\end{proof}

\begin{definition}\label{defgilwehr}
Let $\cat A$ be a small category and let $\posetfunct\colon \cat A\to \bposets$. Following Gillibert and Wehrung~\cite{gillibertwehrung}, we say
that a functor 
\[\liftfunct\colon  \cat A  \to \latasdf\qquad\text{or}\qquad \liftfunct\colon \cat A \to \latf
\]
\emph{lifts the functor $\posetfunct$  with respect to the functor $\fprinc$}, if 
$\posetfunct$ is naturally isomorphic  to the composite functor $\fprinc {} \circ \liftfunct$. 
\end{definition}

Note that the existence of $\liftfunct\colon \cat A \to \latasdf$ above is a stronger requirement than the existence of $\liftfunct\colon \cat A \to \latf$.  
Every ordered set $\tuple{P;\leq}$ can be viewed as a small category whose objects are the elements of $P$ and, for $X,Y\in P$, $|\fmor (X,Y)|=1$ for $X\leq Y$ and $|\fmor (X,Y)|=0$ for $X\nleq Y$. Small categories obtained in this way are called \emph{categorified posets}.
Based on Remark~\ref{remarkepisurjmonoinj}, the known results on representations of monotone maps by principal congruences can be stated in the following two propositions.

\begin{proposition}[\init{G.~}Cz\'edli~\cite{czginjlatcat}]\label{proposinjcatrepr}
Let $\cat A$ be a categorified poset. If $\posetfunct\colon \cat A\to \bposets$ is a functor such that $\posetfunct(f)$ is \emph{injective} for all $f\in\mor{\cat A}$, then there exists a functor 
$\liftfunct\colon  \cat A \to \latasdf$ that lifts $\posetfunct$  with respect to $\fprinc$.
\end{proposition}

Note that \cite{czginjlatcat} extends the result of \init{G.~}Cz\'edli~\cite{czgsingleinjectiveprinc}, in which $\cat A$ is the categorified two-element chain but $F(f)$ is still injective. As another extension of \cite{czgsingleinjectiveprinc}, \init{G.~}Gr\"atzer dropped the injectivity in  the following statement, which we translate to our terminology as follows.

\begin{proposition}[\init{G.~}Gr\"atzer~\cite{gghomoprinc}]\label{propGGonehom}
If $\cat A$ is the categorified two-element chain, then  for every functor $\posetfunct\colon \cat A\to \bposets$, there exists a functor 
$\liftfunct\colon  \cat A \to \latf$ that lifts $\posetfunct$  with respect to $\fprinc$.
\end{proposition}

Equivalently, in a simpler language and using the notation given in \eqref{eqrefsljGfB}, Proposition~\ref{propGGonehom} asserts that  if $X_1$ and $X_2$ are nontrivial bounded ordered sets and $f\colon X_1\to X_2$ is a $\set{0,1}$-preserving monotone map, then there exist lattices $L_1$ and $L_2$ of length 5, order isomorphisms $\tautr_{X_i}\colon \princ{L_i}\to X_i$ for $i\in\set{1,2}$, 
and a $\set{0,1}$-preserving lattice homomorphism $g\colon L_1\to L_2$ such that the diagram
\begin{equation*}
\begin{CD}
\princ{L_1}  @>{ \zetaf g{L_1}{L_2}}>> \princ{L_2} \\
@V{\tautr_{X_1}}VV   @V{\tautr_{X_2}}VV   \\
X_1 @>{f}>>  X_2
\end{CD}
\end{equation*}
is commutative,
that is, $f= \tautr_{X_2} \circ  \zetaf g{L_1}{L_2} \circ \tautr_{X_1}^{-1}$.

Now we are in the position to formulate the second theorem of the paper.

\begin{theorem}\label{thmlat}
Let $\cat A$ be a small category such that 
each $f\in \mor{\cat A}$ is a monomorphism.  Then for every faithful functor $\posetfunct\colon \cat A\to \bposets$,
there exists a faithful functor 
\[
\liftfunct\colon  \cat A \to \latasdf
\]
that lifts $\posetfunct$  with respect to $\fprinc$. Furthermore, if $\posetfunct$ is totally faithful, then there is a totally faithful $\liftfunct$ that lifts $\posetfunct$ with respect to $\fprinc$.
\end{theorem}

Observe that Propositions~\ref{proposinjcatrepr} and \ref{propGGonehom}  are particular cases
Theorem~\ref{thmlat}, since every morphism of a categorified poset is a monomorphism and the functors in these statements are automatically faithful. To avoid the feeling that Proposition~\ref{propGGonehom} or a similar situation is the only case where Theorem~\ref{thmlat} takes care of non-injective monotone maps, we give an example.

\begin{example}\label{exAmplesubcatE} 
Let $D_1,D_2\subseteq \obj{\bposets}$ such that $D_1$ and $D_2$ are sets and $D_1$ is nonempty. We define a 
 small  category $\cat A=\cat A(\bposets, D_1,D_2)$ by the equalities $\obj{\cat A}=D_1\cup D_2$ and 
\begin{equation}
\begin{aligned}
\mor{\cat A}=\{&f\in\fmor_{\bposets}(X,Y): \text{either }X,Y\in D_1\text{ and }f\text{ is a}\cr
&\text{monomorphism in }\bposets,\text{ or } X\in D_2,\,\, Y\in D_1 \}\text.
\end{aligned}
\label{excatDef}
\end{equation}
Then all morphisms of $\cat A$ are monomorphisms in $\cat A$ but, in general, many of them are not injective. The same is true for all subcategories of $\cat A$. (Note that the same holds if we starts from a variety of general algebras rather than from $\bposets$.) By Remark~\ref{remarkepisurjmonoinj}, we can replace ``monomorphism'' by ``injective'' in the second line of \eqref{excatDef}.) Now if $\posetfunct\colon \cat A\to \bposets$ 
is the inclusion functor 
$\eidfunct{ \cat A}\bposets$, see \eqref{eqrinclusionfunctor}, then  Theorem~\ref{thmlat} yields a totally faithful functor 
$\liftfunct\colon  \cat A \to \latasdf$ that lifts $\posetfunct$  with respect to $\fprinc$.
\end{example}

\begin{proof} If $f\circ g_i$ is defined in $\cat A$, then $f$ is a monomorphism in $\bposets$.
\end{proof}

\begin{example}
In a self-explanatory simpler (but less precise) language, we mention two particular cases of Example~\ref{exAmplesubcatE}. First, we 
can represent all automorphisms of a bounded ordered set simultaneously by principal congruences. Second, if we are given two distinct bounded ordered sets $X$ and $Y$, then we can simultaneously represent all $\set{0,1}$-preserving monotone $X\to Y$ maps by principal congruences. 
\end{example}

\section{Quasi-colored lattices and a toolkit for them}\label{secquasicolor}
\subsection{Gadgets and basic facts}
We follow the terminology of \init{G.~}Cz\'edli~\cite{czginjlatcat}. 
If $\nu$ is a quasiorder, that is, a reflexive transitive relation, then $\pair xy\in\nu$ will occasionally be abbreviated as $x\leqnu y$. 
For a lattice or ordered set  $L=\tuple{L;\leq}$ and $x,y\in L$, $\pair x y$ is called an \emph{ordered pair} of $L$ if $x\leq y$. If $x=y$, then  $\pair x y$  is a \emph{trivial ordered pair}. The set of ordered pairs of $L$ is denoted by $\pairs L$. If $X\subseteq L$, then $\pairs X$ will stand for $X^2\cap\pairs L$.
We also need the notation $\covpairs L:=\set{\pair xy\in \pairs X: x\prec y}$. 
By a \emph{quasi-colored  lattice} we mean a structure 
\[\alg L=\tuple{L, \leq;\gamma;H,\nu}
\]
where $\tuple{L;\leq}$ is a lattice, $\tuple{H;\nu}$ is a quasiordered set, $\gamma\colon \pairs L\to H$ is a surjective map, and for all $\pair{u_1}{v_1},\pair{u_2}{v_2}\in \pairs L$, 
\begin{enumerate}[\quad \normalfont({C}1)]
\item\label{labqa} if $\gamma(\pair{u_1}{v_1})\leqnu \gamma(\pair{u_2}{v_2})$, then $\cg{u_1}{v_1}\leq \cg{u_2}{v_2}$ and 
\item\label{labqb} if $\cg{u_1}{v_1}\leq \cg{u_2}{v_2}$, then $\gamma(\pair{u_1}{v_1})\leqnu \gamma(\pair{u_2}{v_2})$.
\end{enumerate}
This concept is taken from  \init{G.~}Cz\'edli~\cite{czginjlatcat}; see  
\init{G.\ }Gr\"atzer, \init{H.\ }Lakser, and \init{E.T.\ }Schmidt~\cite{grlaksersch}, \init{G.\ }Gr\"atzer~\cite[page 39]{grbypict}, and  \init{G.~}Cz\'edli~\cite{czgrepres} and \cite{czgprincc} for the evolution of this concept.
The importance of quasi-colored lattices in the present paper will be made clear in Lemma~\ref{lemmaQCOLisoPrinc} and Corollary~\ref{corolxilemma}.
It follows easily from (C\ref{labqa}), (C\ref{labqb}), and the surjectivity of $\gamma$ that if $\tuple{L,\leq;\gamma;H,\nu}$ is a quasi-colored bounded lattice, then 
$\tuple{H;\nu}$ is a  quasiordered set with a least element and a greatest element; possibly with many least elements and many greatest elements. 
For $\pair x y\in L$, $\gamma(\pair x y)$ is called the \emph{color} (rather than the quasi-color) of $\pair x y$, and we say that $\pair x y$ is colored (rather than quasi-colored) by $\gamma(\pair x y)$. The following convention applies to all of our figures that contain 
thick edges and, possibly, also thin edges: if $\gamma$ is a quasi-coloring, then for an  ordered pair~$\pair x y$, 
\begin{equation}
\gamma(\pair x y)=
\begin{cases}
0,&\text{iff } x=y,\cr
u,&\text{if }x\prec y\text{ is a thin edge labeled by }u,\cr
1,&\text{if the interval }[x,y]\text{ contains is a thick edge,}\cr
\gamma(\pair{x'}{y'}),&\text{if  }[x,y]\text{ and }[x',y']\text{ are transposed intervals.}
\end{cases}
\label{fiGrSgy}
\end{equation}
Based on this convention, our figures determine the corresponding quasi-colorings. 

The quasi-colored lattice 
\[\apupgad 2(p,q) :=\tuple{\pupgad 2(p,q),\lambda\up_{2 p q};\gamma\up_{2 p q};H_2(p,q),\nu_{2 p q}}
\]
in Figure~\ref{fig-uptwogadget},  taken from \init{G.~}Cz\'edli~\cite{czginjlatcat} where it was denoted by $\agad\up(p,q)$,  is our \emph{upward gadget of type $2$}.
Its quasi-coloring is defined by \eqref{fiGrSgy}; note that $\gamma\up_{2 p q}(\pair{\cpq4}{\dpq4})=q$.

\begin{figure}[htc]
\centerline
{\includegraphics[scale=1.0]{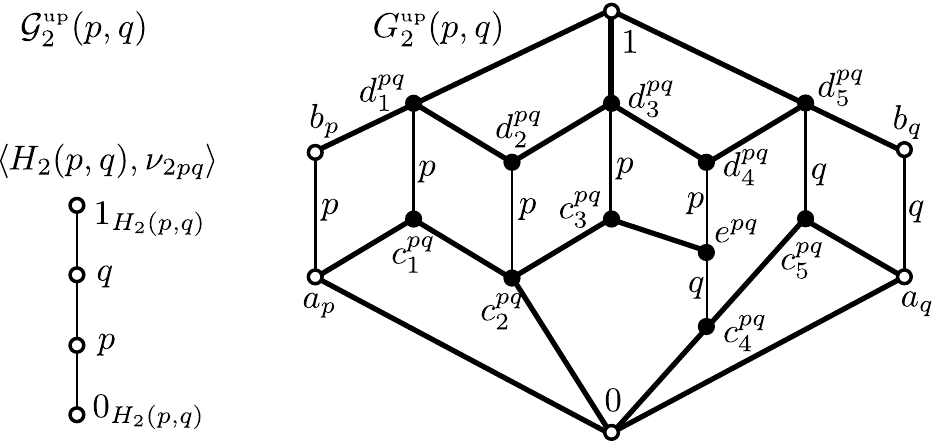}}
\caption{The upward gadget of rank 2\label{fig-uptwogadget}}
\end{figure}

Using the quotient lattices  
\begin{equation}
\begin{aligned}
\pupgad 0(p,q)&:=\pupgad 2(p,q)/\cg{a_q}{b_q}  \text{ and}\cr
\pupgad 1(p,q)&:=\pupgad 2(p,q))/\cg{a_p}{b_p},
\end{aligned}
\label{refeqlowrankgad} 
\end{equation}
we also define the gadgets 
\[\begin{aligned}
\apupgad 0(p,q) &:=\tuple{\pupgad 0(p,q),\lambda\up_{0 p q};\gamma\up_{0 p q};H_0(p,q),\nu_{0 p q}}\text{ and} \cr
\apupgad 1(p,q) &:=\tuple{\pupgad 1(p,q),\lambda\up_{1 p q};\gamma\up_{1 p q};H_1(p,q),\nu_{1 p q}}
\end{aligned}
\]
of rank $0$ and rank $1$, respectively; see Figures~\ref{fig-uponegadget} and  \ref{fig-upnulgadget}. 
Note that the rank is 
$\length{([a_p,b_p])}+\length{([a_q,b_q])}$.
We obtain the \emph{downward gadgets} 
$\apdngad 2(p,q)$, $\apdngad 1(p,q)$, and $\apdngad 0(p,q)$ of ranks $2$, $1$, and $0$ from the corresponding upward gadgets by dualizing; 
see \init{G.~}Cz\'edli~\cite[\dslRtBxB]{czginjlatcat}. Instead of  $\dpq{ij}$ and, if applicable, 
$\cpq{ij}$ and $\epq$, their elements are denoted by  $\ddpq{ij}$,  $\dcpq{ij}$, and $\depq$; see  \cite{czginjlatcat}.  By a \emph{single gadget} we mean an upper or lower gadget. 
The adjective ``upper'' or ``lower'' is the \emph{orientation} of the gadget. 
A single gadget of rank $j$ without specifying its orientation  is denoted by $\pqmgad j (p, q)$.  

\begin{figure}[htc]
\centerline
{\includegraphics[scale=1.0]{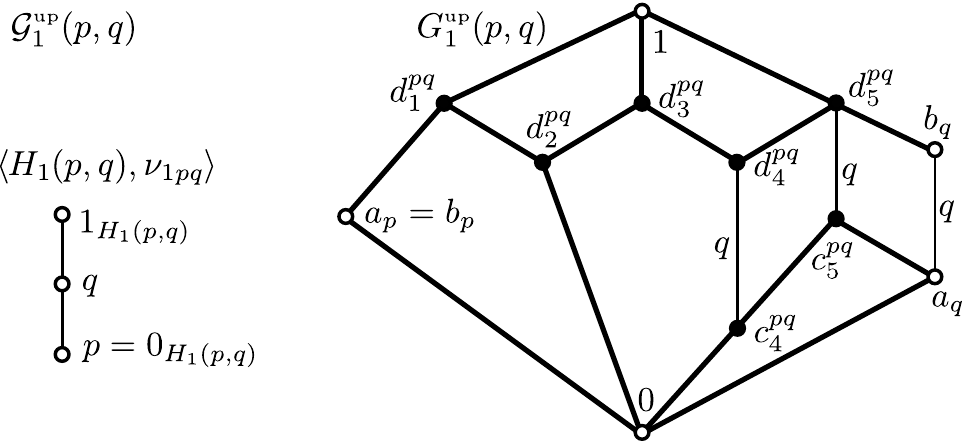}}
\caption{The upward gadget of rank 1  \label{fig-uponegadget}}
\end{figure}

\begin{figure}[htc]
\centerline
{\includegraphics[scale=1.0]{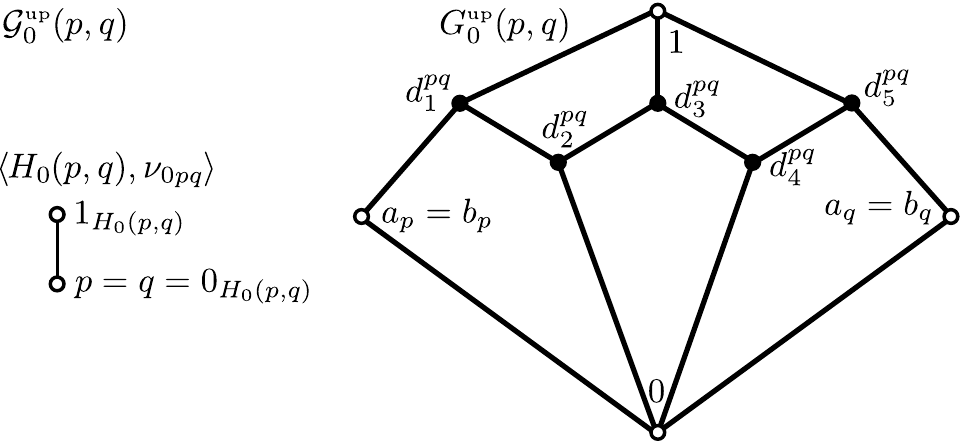}}
\caption{The upward gadget of rank 0    \label{fig-upnulgadget}}
\end{figure}

In case of all  our gadgets $\pqmgad j (p, q)$, we automatically assume that $p\neq q$. Also, we always assume that, for  $i,j\in\set{0,1,2}$, ordered pairs $\pair p q,\pair uv$, and strings $\textup{s}, \textup{t}\in\set{\textup{up}, \textup{dn}}$ 
such that $\tuple{p,q,i,\textup{s}} \neq  \tuple{u,v,j,\textup{t}}$,    
\begin{equation} 
\parbox{6.5cm}{the intersection of  $\pstrgad ti(p,q)$ and   $\pstrgad sj(u,v)$ is as small as it follows from the notation.}
\label{gadgetconV}
\end{equation} 
This convention allows us to form the union  
$\adbgad i (p,q)$ of 
$\apupgad i(p,q)$ and $\apdngad i(p,q)$, for $i\in\set{0,1,2}$, which we call a \emph{double gadget} of rank $i$. While $\adbgad 1 (p,q)$
and $\adbgad 0 (p,q)$ are given in Figures~\ref{fig-dbleonegadget} and \ref{fig-dblenulgadget}, the double gadget $\adbgad 2 (p,q)$ of rank 2 is depicted in \init{G.~}Cz\'edli~\cite[Figure 4]{czginjlatcat}. Observe that all the thin edges are $q$-colored in  $\adbgad 1 (p,q)$ and, in lack of thin edges, all the edges are $1$-colored in $\adbgad 0 (p,q)$. 
For $i\in\set{0,1,2}$,   $\dbgad i(p,q)$ is a selfdual lattice; we will soon point out that 
 $\adbgad i(p,q)$ is a quasi-colored lattice. 

\begin{figure}[htc]
\centerline
{\includegraphics[scale=1.0]{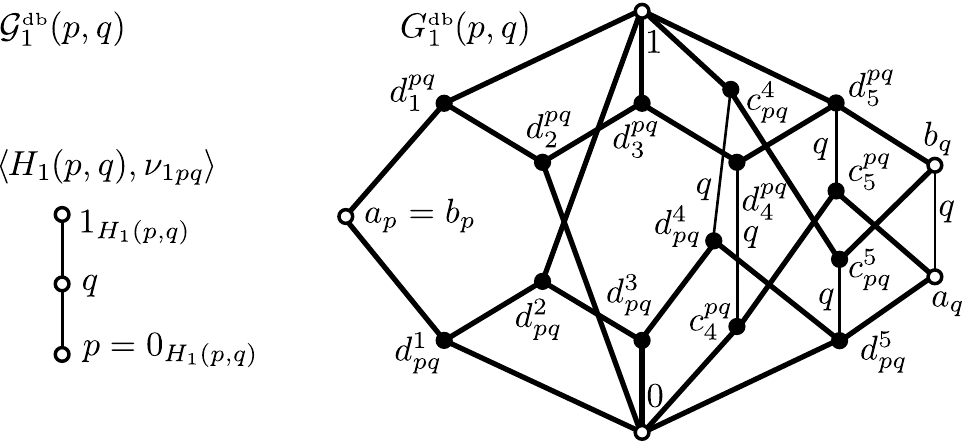}}
\caption{The double gadget of rank 1    \label{fig-dbleonegadget}}
\end{figure}

\begin{figure}[htc]
\centerline
{\includegraphics[scale=1.0]{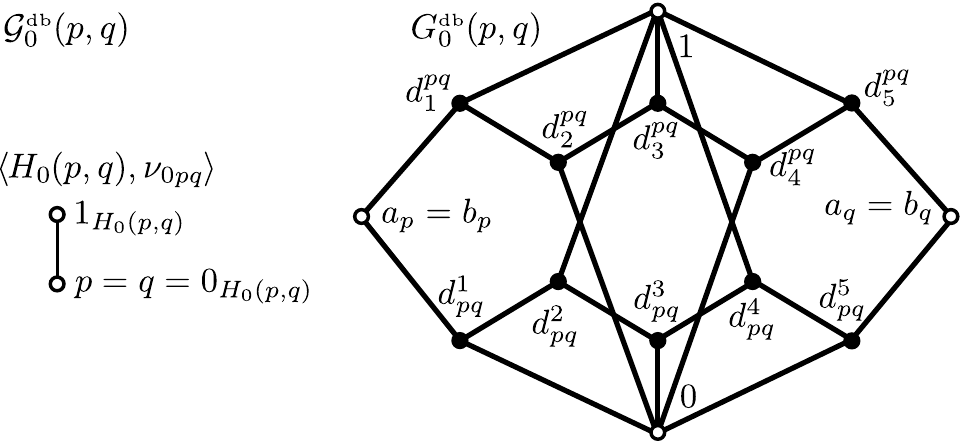}}
\caption{The double gadget of rank 0    \label{fig-dblenulgadget}}
\end{figure}

For $S\subseteq X\times X$, the least quasiorder including $S$ is denoted by $\quos S=\iquos XS$; we write $\quo a b$ rather than $\quos{\set{\pair a b}}$. 

\begin{lemma}\label{hinsgadlemma}
Assume that  $L=\tuple{L; \leq_L}=\tuple{L; \lambda_L}$ is a lattice of length 5,  and let $0<a_p\preceq b_p<1$ and $0<a_q\preceq b_q<1$ in $L$ such that, with $j:= \length([a_p,b_p]) + \length([a_q,b_q])$, 
\allowdisplaybreaks{
\begin{align*}
a_p\vee_L a_q=1\quad b_p\wedge_L b_q=0,\quad 
L\cap \pupgad j(p,q)=\set{0,a_p, b_p, a_q, b_q,1},  \cr
0\leq \length([a_p,b_p]) \leq \length([a_q,b_q])\leq 1, \quad \length([0,a_p])\leq 2, \cr 
\length([b_p,1])\leq 2, \quad \length([0,a_q])\leq 2,\quad\text{and}\quad \length([b_q,1])\leq 2\text.
\end{align*}}%
Let  
\begin{equation*}
\begin{aligned}
\uins L := L\cup \pupgad j(p,q)\,\text{ and }\, \uins\lambda:= \quos{\lambda_L\cup\lambda\up_{j p q}};
\end{aligned}
\end{equation*}
see \cite[Figure \figinsgad]{czginjlatcat} for $j=2$.
Then $\uins L= \tuple{\uins L;\uins\lambda}$, also denoted by $\upins L j p q$ or $\tuple{\upins L j p q; \lequins}$,
is  a lattice of length 5. Also,  both $L$ and $\pupgad j(p,q)$ are $\set{0,1}$-sublattices of~$\uins L$.
\end{lemma}

\begin{definition}\label{defClGadg}
Within $\uins L$, the (sublattice) $\pupgad j(p,q)$ is the \emph{upper gadget from $\pair{a_p}{b_p}$ to  $\pair{a_q}{b_q}$}. By duality, we can analogously glue the \emph{lower gadget} $\pdngad j(p,q)$ into $L$ \emph{from} $\pair{a_p}{b_p}$ \emph{to}  $\pair{a_q}{b_q}$. Applying Lemma~\ref{hinsgadlemma}, its dual, and \eqref{gadgetconV}, we can  glue the \emph{double gadget} $\dbgad j(p,q)$ into $L$ \emph{from} $\pair{a_p}{b_p}$ \emph{to}  $\pair{a_q}{b_q}$.
\end{definition}

\begin{proof}[Proof of Lemma~\ref{hinsgadlemma}] For $j=2$, the lemma coincides with \cite[Lemma~\insgadlemma]{czginjlatcat} while the case $j<2$ is analogous but simpler. Hence, it would suffice to say that the proof in \cite{czginjlatcat} works without any essential modification. However, since we will need some formulas from the proof later, we give some details. To simplify our equalities below,  we denote $\pupgad j(p,q)$
by $\pupgad j$ and, in subscript position,  by $\gad$.
As in \cite{czginjlatcat},  
we can still use the sublattice
\[B=B(p,q):=\set{0,a_p,b_p,a_q,b_q,1}=L\cap\pupgad j(p,q),
\]
the closure operators 
\begin{equation*}
\begin{aligned}
& \fcs{\phantom x}\colon \pupgad j\to B\text{, where } \fcs x\text{ is the smallest element of }B  \cap \ifilter {\gad} x, \cr
& \fhat{\phantom x}\colon L \to B\text{, where }  \fhat x\text{ is the smallest element of }B \cap \ifilter L x,
\end{aligned}
\end{equation*}
and, dually,  the interior operators
\begin{equation*}
\begin{aligned}
& \acs{\phantom x}\colon \pupgad j\to B\text{, where } \acs x\text{ is the largest element of }B \cap \iideal {\gad} x,\cr
& \ahat{\phantom x}\colon L \to B\text{, where } \ahat x\text{ is the largest element of }B \cap \iideal L x; 
\end{aligned}
\end{equation*}
which were introduced in  
 \cite[\cloops{} and \intops]{czginjlatcat}. 
Since our gadgets are "wide enough" in some geometric sense, the operators above are well-defined. As in \cite[\nRlDeF]{czginjlatcat}, 
\begin{equation*}
\begin{aligned}
&\uins\lambda \text{ is an ordering, }
\restrict { \uins\lambda}{L}=\lambda_L,\quad\restrict { \uins\lambda}{\gad}=\lambda\up_{p q},\cr
&\text{for }x\in L\text{ and }y\in \pupgad j,\,\,\,
x\lequins  y\iff \fhat x \leq_{\gad} y \iff x\leq_L \acs y, 
\cr
&\text{for }x\in \pupgad j\text{ and }y\in L,\,\,\,
x\lequins y\iff \fcs x \leq_L y \iff x\leq_{\gad} \ahat y\text. 
\end{aligned}
\end{equation*}
Denote the lattice operations in $L$ and $\pupgad j$ by $\vee_L$,  $\wedge_L$, and $\vee_{\gad}$, $\wedge_{\gad}$, respectively.  For $x,y\in \uins L$, we have that
\begin{align}
&\text{if }x\in L\setminus \pupgad j\text{ and }y\in \pupgad j\setminus L,\text{ then }x\wedgeuins y=  x  \wedge_L \acs y, 
\label{insgada} \\
&\text{if }x\in L\setminus \pupgad j\text{ and }y\in \pupgad j\setminus L,\text{ then }x\veeuins y= \fhat x \vee_{\gad} y, 
\label{insgadb} \\
&\text{if }x, y\in L ,\text{  then }x\wedgeuins y=  x  \wedge_L y\text{, and } x\veeuins y=  x  \vee_L y, 
\label{insgadc}\\
&\text{if }x, y\in \pupgad j,\text{  then }x\wedgeuins y=  x  \wedge_{\gad} y\text{, and } x\veeuins y=  x  \vee_{\gad} y\text.\label{insgadd}
\end{align}
These equations, which are  \cite[\insgada--\insgadd]{czginjlatcat} for $j=2$ and which are proved by exactly the same argument for $j<2$, show that $\uins L$ is a lattice. 
\end{proof}

\begin{figure}[htc]
\centerline
{\includegraphics[scale=1.0]{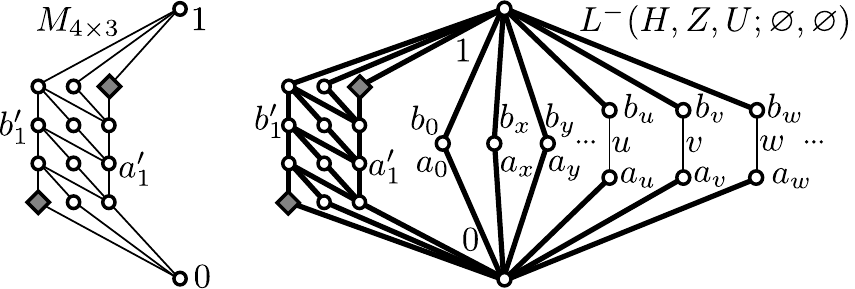}}
\caption{$\Mnh$  and $\Nla^-(H,Z,U; \emptyset,\emptyset)$, which is \emph{not} quasi-colored\label{fig-m33}}
\end{figure}

\subsection{Large lattices}
Let $H$ be a set and $Z,U\subset H$ such that 
\begin{equation}
0\in Z,\quad 1\in U,\quad\text{and}\quad Z\cap U=\emptyset\text.
\label{conventZH}
\end{equation} 
The selfdual simple lattice on the left of   Figure~\ref{fig-m33} is denoted by $\Mnh$; see also \cite[Figure \figmthreethree]{czginjlatcat} for another diagram. 
(The two square-shaped gray-filled elements will play a special role later.)
Also, we denote by 
\[\Nla^-(H,Z,U; \emptyset,\emptyset)=\tuple{\Nla^-(H,Z,U; \emptyset,\emptyset); \lambda_{\Nla^-(H,Z,U; \emptyset,\emptyset)} }
\]  
the lattice on the right, where $Z=\set{0, x,y\dots}$ and $H\setminus Z=\set{u,v,w,\dots}$. This lattice is almost the same as that on the right of \cite[Figure \figmthreethree]{czginjlatcat}. Note, however, that $|Z|$ and $|U|$ can be arbitrarily large cardinals. Note also that for $z\in Z$, $a_z=b_z$. The role of $\Mnh$ in the construction is two-fold. First, it is a simple lattice and it guarantees that all the thick edges are 1-colored, that is, they generate the largest congruence, even if $|H|=2$. Second, $\Mnh$ guarantees that $\Nla^-(H,Z,U; \emptyset,\emptyset)$ is of length 5. 
Since  $\pair{a_1}{b_1}$ is  1-colored according to labeling  but this edge does not generate the largest congruence, $\Nla^-(H,Z,U; \emptyset,\emptyset)$ is not a quasi-colored lattice (at least, not if $1$ is intended to be a largest elements in $H$). So  we cannot be satisfied yet. In order to make this edge and all the $\pair{a_r}{b_r}$, for $r\in U$, generate the largest congruence, Definition~\ref{defClGadg} allows us   
\begin{equation*}
\parbox{8.0cm}{to glue, for each $r\in U$, a distinct copy of $\dbgad 2(p,q)$ into $\Nla^-(H,Z,U; \emptyset,\emptyset)$ from $\pair{a_1'}{b_1'}$ 
to $\pair{a_r}{b_r}$.}
\end{equation*}
(No matter if we  glue the gadgets one by one by a transfinite induction or glue them simultaneously, we obtain the same.)
It follows from Lemma~\ref{hinsgadlemma} that we obtain a lattice in this way; we denote this lattice by 
\[\Nla(H,Z,U; \emptyset,\emptyset)=\tuple{\Nla(H,Z,U; \emptyset,\emptyset); \lambda_{H,Z,U; \emptyset,\emptyset} }\text.
\]
Note at this point that, after adding these gadgets to $\Nla^-(H,Z,U; \emptyset,\emptyset)$, 
\begin{equation}
\text{all edges of these gadgets become thick;}
\label{eqrgeathick}
\end{equation}
see convention ~\eqref{fiGrSgy}. Let 
\[\nu_{H,Z,U; \emptyset,\emptyset}=\quos{(Z\times H) \cup (H\times U)},
\]
and define $\gamma_{H,Z,U; \emptyset,\emptyset}$ by convention \eqref{fiGrSgy}. 
It is straightforward to see that 
\begin{equation}
\begin{aligned}
\alg\Nla(&H,Z,U; \emptyset,\emptyset)= \cr
&\tuple{\Nla(H,Z,U; \emptyset,\emptyset),
\lambda_{H,Z,U; \emptyset,\emptyset}; 
\gamma_{H,Z,U; \emptyset,\emptyset};
H, \nu_{H,Z,U; \emptyset,\emptyset} 
}
\end{aligned}
\end{equation}
is a quasi-colored lattice.

Next, to obtain larger lattices, we are going to insert  gadgets into the lattice $\Nla(H,Z,U; \emptyset,\emptyset)$ in a certain way. It will prompt follow Lemma~\ref{hinsgadlemma} that we obtain lattices; in particular, $\lambda_{H,Z,U;I,J}$ in \eqref{LAIJlat} will be a lattice order. Assume that
\begin{equation}
\parbox{8cm}{ $I$ and $J$ are subsets of $H\times H$ such that $p\neq q$ and ($q\in Z\then p\in Z$) hold for every $\pair p q\in I\cup J$.}
\label{IJassmpTion}
\end{equation}
With this assumption, we define the \emph{rank of a pair} $\pair p q\in I\cup J$ as follows:
\[ \rank{\pair p q}:=
\begin{cases}
0,&\text{if } p,q\in Z, \cr
1,&\text{if } p\in Z\text{ and }q\in H\setminus Z,\cr
2,&\text{if } p,q\in H\setminus Z\text.
\end{cases}
\]
Let us agree that, for every $\pair p q\in I\cup J$ and $j:=\rank{\pair p q}$, 
\begin{equation}
\begin{aligned}
 \pupgad j(p,q)\cap \Nla(H,Z,U;\emptyset,\emptyset)&=\set{0,a_p,b_p,a_q,b_q,1}\text { and}   \cr
\pdngad j(p,q)\cap \Nla(H,Z,U;\emptyset,\emptyset)&=\set{0,a_p,b_p,a_q,b_q,1}\text.
\end{aligned}
\label{NAconV}
\end{equation}
Taking  Conventions \eqref{gadgetconV} and \eqref{NAconV} into account, we define
\begin{equation}
\begin{aligned}
\Nla(H,Z,U;I, J) &:= \Nla(H,Z,U;\emptyset,\emptyset)\cup\bigcup_{\pair p q\in I}\pupgad{\rank{\pair p q}} (p,q) \cr
& \kern 2.0cm\cup \bigcup_{\pair p q\in J}\pdngad{\rank{\pair p q}} (p,q),\quad\text{and}\cr
\lambda_{H,Z,U;I,J} &:= 
\Quos{ \lambda_{H,Z,U;\emptyset,\emptyset}\cup\bigcup_{\pair p q\in I}\lambda\up_{p q} \cup \bigcup_{\pair p q\in J}\lambda\dn_{p q} } \text.\cr
\end{aligned}
\label{LAIJlat}
\end{equation} 
Based on Lemma~\ref{hinsgadlemma} and its dual, a trivial transfinite induction yields that 
\[\Nla(H,Z,U;I, J)=\tuple{\Nla(H,Z,U;I, J); \lambda_{H,Z,U;I,J}}
\]
is a lattice of length 5. Clearly, if $I=J$, then this lattice is selfdual. Let us emphasize that whenever we use the notation $\Nla(H,Z,U;I, J)$,  \eqref{IJassmpTion} is assumed.%
\begin{remark}\label{remstrongfaithfconstr}
For later reference, we note that for lattices of the form \eqref{LAIJlat},
$a_p$, $b_p$, $\cpq {ij}$, $\dpq{ij}$, $\dcpq{ij}$, etc.\ are understood as abbreviations for $\tuple{a,p}$, $\tuple{b,p}$, $\tuple{c,p,q,i,j}$, $\tuple{d,p,q,i,j}$, $\tuple{c^{\text{dual}},p,q,i,j}$, etc.. Therefore, 
\begin{align*}
\Nla(H_1,Z_1,U_1;I_1, J_1)= \Nla(H_2,Z_2,U_2;I_2, J_2)\text{ iff }\cr
\tuple{H_1,Z_1,U_1,I_1, J_1}= \tuple{H_2,Z_2,U_2,I_2, J_2}\text.
\end{align*}
\end{remark}

\subsection{Large quasi-colored lattices}
Assuming \eqref{conventZH}, let $\inter H:=H\setminus(Z\cup U)$. Also, let $\nu_{H,Z,U;\emptyset,\emptyset}=\quos{(Z\times H)\cup (H\times Z)}$. Note that each $z\in Z$ is a least element of $\tuple{H; \nu_{H,Z,U;\emptyset,\emptyset}}$ and each $u\in U$ is a largest element. Also,  for any two distinct $p,q\in \inter H$, $p$ and $q$ are incomparable, that is, none of $\pair p q$ and $\pair q p$ belongs to $\nu_{H,Z,U;\emptyset,\emptyset}$.
With convention \eqref{IJassmpTion}, let 
\begin{equation*}
\begin{aligned}
\nu_{H,Z,U;I,J}&:=\iquos H{\nu_{H,Z,U;\emptyset,\emptyset}\cup I\cup J }
\cr
&\kern 2.5 pt =\quos{(Z\times H)\cup (H\times Z)\cup I\cup J}
\text.
\end{aligned}
\end{equation*}
Based on \eqref{NAconV}, it is easy to see that  
\begin{equation}
\gamma_{H,Z,U;I,J} := 
 \gamma_{H,Z,U;\emptyset,\emptyset}\cup\bigcup_{\pair p q\in I}\gamma\up_{\rank{\pair p q} p q} \cup \bigcup_{\pair p q\in J}\gamma\dn_{\rank{\pair p q} p q}  
\label{eqrgammahzuIJ}
\end{equation}
is a well-defined map from $\pairs{\Nla(H,Z,U;I,J)}$ to $H$.

\begin{lemma}\label{keylemma}
Assume \eqref{IJassmpTion}. Then
\begin{equation}
\begin{aligned}
\alg\Nla(&H,Z,U;I,J)\cr
&:=   
\tuple{\Nla(H,Z,U;I,J),  \lambda_{H,Z,U;I,J};\gamma_{H,Z,U;I,J};  H, \nu_{H,Z,U;I,J} }
\end{aligned}
\label{nLfJbXY}
\end{equation} 
is a quasi-colored lattice of length 5. If $I=J$, then it is a selfdual lattice.
\end{lemma}

\begin{proof} 
If $Z=\set 0$ and $\rank{\pair p q}=2$ for all $\pair p q\in I\cup J$, then the statement is practically the same as \cite[Lemma~\keylemma]{czginjlatcat}. (Although $1\notin U=\emptyset$ in \cite[Lemma~\keylemma]{czginjlatcat}, this does not make any difference.) 
The nontrivial part is to show 
(C\ref{labqb}). This argument in \cite{czginjlatcat} has two ingredients, and these ingredients also work in the present situation.

First, let  $\balpha$ be the equivalence on $\Nla(H,Z,U;I,J)$ whose non-singleton equivalence classes are the $[a_p,b_p]$ for $p\in \inter H$, 
the $[\cpq i,\dpq i]$ for $\pair p q\in I$ and $i\in\set{1,\dots,5}$, and the $[\ddpq i,\dcpq i]$ for $\pair p q\in J$ and $i\in\set{1,\dots,5}$. Using the Technical Lemma from \init{G.~}Gr\"atzer~\cite{ggtechnicallemma}, cited in \cite[Lemma~\technicallemma]{czginjlatcat},  it is straightforward to see that $\balpha$ is a congruence. Clearly, $\balpha$ is distinct from $\nablaell{\Nla(H,Z,U;I,J)}$, the largest congruence of $\Nla(H,Z,U;I,J)$. Like in \cite[\nablIfffcl]{czginjlatcat}, this implies easily that, for any $\pair x y\in\pairs{\Nla(H,Z,U;I,J)}$, 
\begin{equation*}
\gamma_{H,Z,U;I,J}(\pair x y) = 1\iff \cg x y = \nablaell{\Nla(H,Z,U;I,J)}\text.
\end{equation*} 

The second ingredient is to show that 
\begin{equation}
\parbox{10 cm}{if $p,q\in\inter H$, $\cg{a_p}{b_p} \leq   \cg{a_q}{b_q}\neq\nablaell{\Nla(H,Z,U; I,J)}$, and
$p\neq q$,  then 
$\pair p q=\pair{\gamma_{H,Z,U;I,J}(\pair{a_p}{b_p})}{  \gamma_{H,Z,U;I,J}(\pair{a_q}{b_q})}\in \nu_{H,Z,U;I,J}$;}
\label{dwhTppwhjS}
\end{equation}
compare this with \cite[\dwhTppwhjS]{czginjlatcat}. 
The inequality $\cg{a_p}{b_p} \leq   \cg{a_q}{b_q}$ is equivalent to the containment $\pair{a_p}{b_p} \in   \cg{a_q}{b_q}$. This containment is witnessed by a \emph{shortest} sequence of consecutive prime intervals in the sense of the Prime-projectivity Lemma of \init{G.~}Gr\"atzer~\cite{ggprimeprojective}; note that this lemma is cited in \cite[Lemma~\primprojlemma]{czginjlatcat}. If one of the prime intervals in the sequence generates $\nablaell{\Nla(H,Z,U; I,J)}$, then the easy direction of the Prime-projectivity Lemma yields that  $\cg{a_q}{b_q}=\nablaell{\Nla(H,Z,U; I,J)}$, a contradiction. Hence, none of these prime intervals generates $\nablaell{\Nla(H,Z,U; I,J)}$. Thus, since (C\ref{labqa}) is easily verified in the same way as in \cite{czginjlatcat}, none of these prime intervals is $1$-colored. In other words, all prime intervals of the sequence are thin edges. Gadgets of rank 0 contain no thin edges, so the sequence avoids them. The same holds for the gadgets mentioned in \eqref{eqrgeathick}. 
Gadgets of rank 1 contain too few thin edges, so the sequence can only make a loop in them; this is impossible since we consider the shortest sequence. Thus, the sequence goes in the sublattice that we obtain by omitting all gadgets of rank less than 2, all gadgets occurring in \eqref{eqrgeathick}, and all elements $a_z=b_z$ for $z\in Z$. So we can work in this sublattice, which is the same as the lattice considered in \cite[\dwhTppwhjS]{czginjlatcat}. Consequently, the proof  of  \cite[\dwhTppwhjS]{czginjlatcat} yields \eqref{dwhTppwhjS}. 
\end{proof}

\section{From quasiorders to homomorphisms}\label{secfromquasitohomo}
For a quasiordered set $\tuple{H;\nu}$, we define
\begin{equation}
\begin{aligned}
Z(H)&:=\set{x\in H: (\forall y\in H)\,\,(\pair x y\in\nu)}\text{ and}\cr
U(H)&:=\set{x\in H: (\forall y\in H)\,\,(\pair y x\in\nu)}\text.
\end{aligned}
\label{eqrZHUHdf}
\end{equation}
These are the set of \emph{smallest elements} (the notation comes from ``zeros'') and that of \emph{largest elements} (``units''). In this section, we are only interested in the following particular case of the quasi-colored lattices $\alg\Nla(H,Z,U;I,J)$.

\begin{definition}\label{defLofquord}
For a quasiordered set $H=\tuple{H;\nu}$, assume that
\begin{equation}
0\in Z(H),\quad 1\in U(H),\quad\text{and}\quad 
0\neq 1\text{.}
\label{eqrHnuassum}
\end{equation}
With this assumption,  we define 
\begin{equation}
\begin{aligned}
\alg\Nla(H,\nu)=  \tuple{\Nla(H,\nu),  \lambda_{H,\nu};\gamma_{H,\nu};  H, \nu}\,\text{ as }\,\alg\Nla(H,Z(H), U(H);\nu,\nu)
\end{aligned}
\label{eqrLHnudf}
\end{equation}
according to \eqref{nLfJbXY}. Note that $\nu=\nu_{H,Z(H),U(H);\nu,\nu}$ and, clearly, $\Nla(H,\nu)$ is a selfdual lattice of length 5.
\end{definition}

For quasiordered sets $\tuple{H_1;\nu_1}$ and $\tuple{H_2;\nu_2}$, a map $f\colon H_1\to H_2$ is \emph{monotone} if $\pair xy\in \nu_1$ implies $\pair{f(x)}{f(y)}\in \nu_2$ for all $x,y\in H_1$. Now, we are in the position to state the main lemma of this subsection.

\begin{lemma}\label{lemmahomomain}
Let $\tuple{H_1;\nu_1}$ and $\tuple{H_2;\nu_2}$ be quasiordered sets, both with $0$ and $1$ such that $0\neq 1$. 
If $f\colon H_1\to H_2$ is an \emph{injective} monotone map such that $f(Z(H_1))\subseteq Z(H_2)$ and $f(U(H_1))\subseteq U(H_2)$, then there exists a \emph{unique} $\set{0,1}$-preserving lattice homomorphism $g\colon \Nla(H_1,\nu_1) \to \Nla(H_2,\nu_2)$ such that 
\begin{equation}
g(a_p)=a_{f(p)}\,\text{ and }\, g(b_p)=b_{f(p)},\,\text{ for all }\,p\in H_1, 
\label{eqrefgprops}
\end{equation}
and the $g$-image of each of the two square-shaped gray-filled elements, see Figure~\ref{fig-m33}, is a square-shaped gray-filled element.   
\end{lemma}

With reference to \eqref{eqrZHUHdf}, note that  $0\in Z(H_i)$, $1\in U(H_i)$, and $Z(H_i)\cap U(H_i)=\emptyset$ hold for $i\in\set{1,2}$.
The assumption of injectivity cannot be omitted from this lemma, because if $f$ is not injective and a $\set{0,1}$-preserving homomorphism  $g$ satisfies \eqref{eqrefgprops}, then the kernel of $g$ collapses some $a_p\neq a_q$, so this kernel is the largest congruence, contradicting $g(0)=0\neq 1=g(1)$.

\begin{proof}[Proof of Lemma~\ref{lemmahomomain}]
First, we deal with the uniqueness of $g$. Since $g(0)=0\neq 1=g(1)$, the kernel congruence $\ker(g)$ of $g$ cannot collapse a thick (that is, 1-colored) edge. Since all edges of $\Mnh$ are thick, the restriction $\restrict g{\Mnh}$ of $g$ to $\Mnh$ is injective. Since no other sublattice of $L_2$ than $\Mnh$ itself is isomorphic to $\Mnh$, it follows that $g(\Mnh)$ is the unique $\Mnh$ sublattice of $\Nla(H_2;\nu_2)$. Observe that except for the two doubly irreducible atoms and the two doubly irreducible coatoms, each element of $\Mnh$ is a fixed point of all automorphisms of $\Mnh$. Therefore, since $g$ preserves the ``square-shaped gray-filled'' property, we conclude that $\restrict g{\Mnh}$ is uniquely determined. The $g$-images of the $a_p$ and $b_p$, $p\in H_1$, are determined by the assumption on $g$.  Observe that an upper gadget $\pupgad 2(p,q)$ has exactly two non-trivial congruences, $\cg{a_p}{b_p}$ and $\cg{a_q}{b_q}$; $\pupgad 1(p,q)$ has only $\cg{a_q}{b_q}$, and $\pupgad 0(p,q)$  has none. The same holds for lower gadgets. Therefore, since $\ker(g)$ cannot collapse a thick edge, it follows easily that the restriction of $g$ to any gadget is uniquely determined. Therefore, $g$ is unique.

In the rest of the proof, we intend to show the existence of $g$. We will define an appropriate $g$ as the union of some partial maps. Let $g_{\Mnh}$ denote the unique isomorphism from the $\Mnh$ sublattice of $\Nla(H_1,\nu_1)$  onto the $\Mnh$ sublattice of  $\Nla(H_2,\nu_2)$ such that $g_{\Mnh}$ preserves the ``square-shaped gray-filled'' property. 
For $i\in\set{1,2}$, we denote $\nu_i\setminus\set{\pair x x: x\in H_i}$ by $\nu_i^+$.
Next, let $\pair p q\in\nu_1^+$ and  $j:=\rank{\pair p q}$. 
By the construction of $\Nla(H_1,\nu_1)$, see \eqref{eqrLHnudf} and Definition~\ref{defClGadg}, the gadget $\pupgad j (p,q)$ is a $\set{0,1}$-sublattice of $\Nla(H_1,\nu_1)$ from $\pair{a_p}{b_p}$ to $\pair{a_q}{b_q}$. Let $p'=f(p)$, $q'=f(q)$, and $j'=\rank{\pair {p'} {q'}}$. Besides that $f$ is monotone, we frequently need the assumption that it is injective; at present, we conclude $\pair{p'}{q'}\in\nu_2^+$ from these assumptions. (Later, we will not always emphasize similar consequences of these assumptions.) 
It follows from  $\pair{p'}{q'}\in\nu_2^+$ and the construction of $\Nla(H_2,\nu_2)$ that  $\pupgad {j'} (p',q')$ is a gadget in $\Nla(H_2,\nu_2)$ from $\pair{a_{p'}}{b_{p'}}$ to $\pair{a_{q'}}{b_{q'}}$.
We obtain  from $f(Z(H_1))\subseteq Z(H_2)$ that
\begin{equation}
j'\leq j\text. 
\label{eqrjleqjprime}
\end{equation}
According to \eqref{refeqlowrankgad}, we can take the unique surjective $\set{0,1}$-preserving lattice homomorphism $\pupmap p q\colon  \pupgad j (p,q)\to  \pupgad {j'} (p',q')$ such that $\pupmap p q(a_p)=a_{p'}$,  $\pupmap p q(b_p)=b_{p'}$,  $\pupmap p q(a_q)=a_{q'}$, and  $\pupmap p q(b_q)=b_{q'}$. We take the $\set{0,1}$-preserving lattice homomorphism  $\pdnmap p q\colon  \pdngad j (p,q)\to  \pdngad {j'} (p',q')$ analogously. 
Note that $g_{\Mnh}$ maps $a_1'\in \Nla(H_1,\nu_1)$ onto  $a_1'\in \Nla(H_2,\nu_2)$, and the same is true for $b_1'$. For $u\in U(H_1)$, we know that $f(u)\in U(H_2)$. By construction, there is an upper gadget of rank 2 from $\pair{a_1'}{b_1'}$ to $\pair{a_u}{b_u}$ in $\Nla(H_1,\nu_1)$, and we have an  upper gadget of rank 2 from $\pair{a_1'}{b_1'}$ to $\pair{a_{f(u)}}{b_{f(u)}}$ in  $\Nla(H_2,\nu_2)$. The unique isomorphism from the first gadget to the second such that $a_1'\mapsto a_1'$, $b_1'\mapsto b_1'$, $a_u\mapsto a_{f(u)}$, and $b_u\mapsto b_{f(u)}$ is denoted by $\pupmap {1'}u$. 
We define the isomorphism $\pdnmap {1'}u$ between the corresponding lower gadgets similarly.
For $\pair{p_1}{q_1}, \pair{p_2}{q_2} \in \nu_1^+$ and $u\in U(H_1)$,  any two of the homomorphisms $g_{\Mnh}$, $\pupmap{p_1}{q_1}$, $\pdnmap{p_1}{q_1}$, $\pupmap{p_2}{q_2}$, $\pdnmap{p_2}{q_2}$, $\pupmap {1'}u$, and $\pdnmap {1'}u$ agree on the intersection of their domains. Therefore,
\begin{equation*} 
g:= g_{\Mnh} \cup \bigcup_{\pair p q \in \nu_1^+} \minusskip \pupmap p q \plusskip \cup 
\bigcup_{\pair p q \in \nu_1^+}\minusskip   \pdnmap p q\plusskip 
\cup \bigcup_{u\in U(H_1)}\minusskip    \pupmap {1'}u\plusskip 
\cup \bigcup_{u\in U(H_1)}\minusskip    \pdnmap {1'}u
\end{equation*}
is a well-defined $\set{0,1}$-preserving map from $\Nla(H_1,\nu_1)$ to $\Nla(H_2,\nu_2)$.

\begin{figure}[htc]
\centerline
{\includegraphics[scale=1.0]{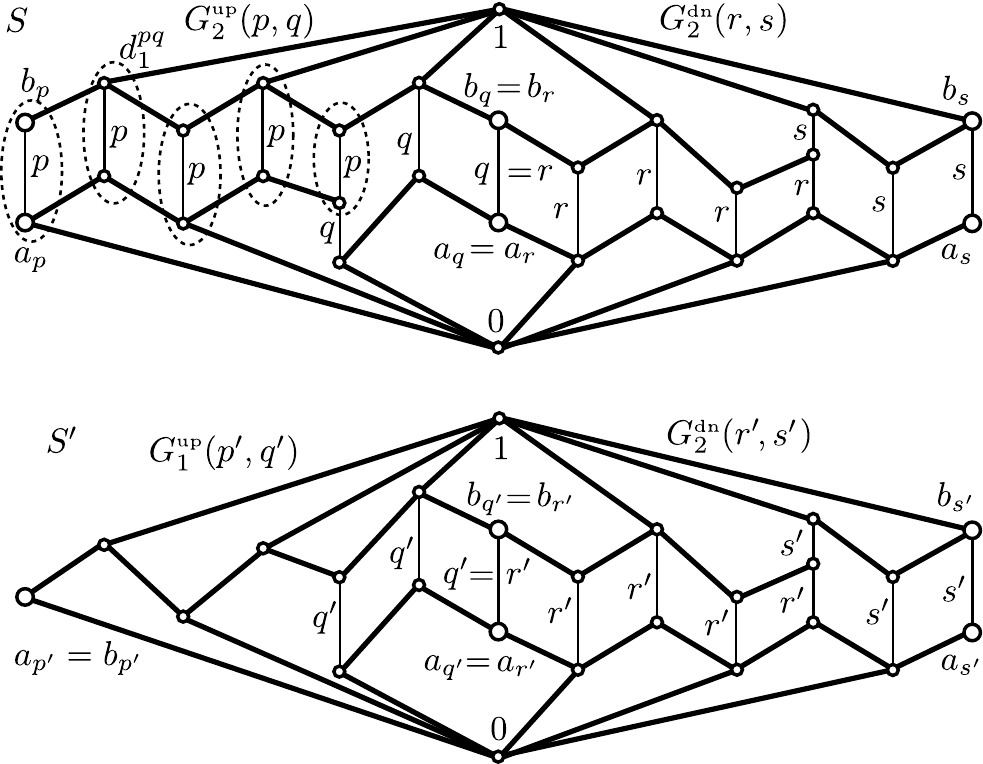}}
\caption{$\tuple{\text{up}, \text{dn}}$, $q=r$, and $\tuple{j,j',k,k'}=\tuple{2,1,2,2}$ \label{fighomopqr}} 
\end{figure}

Next, we are going to show that, for all $x,y\in\Nla(H_1;\nu_1)$,
\begin{equation}
g(x\vee y)=g(x)\vee g(y)\,
\text{ and }\,g(x\wedge y)=g(x)\wedge g(y)\text. 
\label{eqrgJoinhomo}
\end{equation}
Clearly, we can assume that $\set{x,y}\cap \Mnh=\emptyset$ and none of the single gadgets contains both $x$ and $y$. Therefore,  $\set{0,1}\cap\set{x,y}=\emptyset$ and there are single gadgets $\pqmgad j (p, q)$ and $\pqmgad k(r, s)$ containing $x$ and $y$, respectively. 
Of course, $p\neq q$ and $r\neq s$; however, we do not know more than $|\set{p,q,r,s}|\in \set{2,3,4}$.
We can work in the union $S:=\pqmgad j (p, q)\cup \pqmgad k(r,s)$, which is a sublattice by \eqref{insgada}--\eqref{insgadd}; see also the upper parts of Figures~\ref{fighomopqr},  \ref{fighomopqrsupup}, and \ref{fighomogrc}. Let $p':=f(p)$, $q':=f(q)$, $r':=r(p)$,  $s':=f(s)$, and $S':=\pqmgad {j'} (p', q')\cup \pqmgad {k'}(r',s')$; see the lower parts of Figures~\ref{fighomopqr},  \ref{fighomopqrsupup}, and \ref{fighomogrc}, where $g(x)$ is geometrically below $x$ for every $x\in S$. Again, $S'$  is a sublattice by \eqref{insgada}--\eqref{insgadd}.  
(Note that if $\pair p q$ or $\pair r s$  is of the form $\pair {1'}u$ with $u\in U(H_1)$, then we have to extend $f$ by $1'\mapsto 1'$, since $1'\notin H_1$.) 
Let $j:=\rank{\pair p q}$, $k:=\rank{\pair r s}$, $j':=\rank{\pair {p'} {q'}}$, and $k':=\rank{\pair {r'} {s'}}$. 

We know from \eqref{eqrjleqjprime} that $j'\leq j$ and, similarly, $k'\leq k$. Hence, by the definition of our gadgets of rank less than 2, there 
are congruences $\balpha_1$ and $\balpha_2$ of $\pqmgad j (p, q)$ and $\pqmgad k(r,s)$ and surjective 
homomorphisms (namely, the natural projections) $g_1\colon \pqmgad j (p, q) \to \pqmgad {j'} (p', q')$ and $g_2\colon \pqmgad k(r,s) \to  \pqmgad {k'}(r',s')$
such that $\balpha_1$ is the kernel of $g_1$ and $\balpha_2$ is the kernel of $g_2$. In Figures~\ref{fighomopqr}, \ref{fighomopqrsupup}, and \ref{fighomogrc}, the nontrivial 
congruence blocks are indicated by dotted lines. 

By the definition of $g$,  $g_1\cup g_2$ is the restriction $\restrict gS$ of $g$ to $S$. Thus, to verify \eqref{eqrgJoinhomo}, we have to show that $g_1\cup g_2\colon S\to S'$ is a homomorphism. It suffices to show that $\balpha_1\cup \balpha_2$ is a congruence of $S$, because then $S'$ is the quotient lattice of $S$ modulo $\balpha_1\cup \balpha_2$ and $g_1\cup g_2$ is the natural projection homomorphism of $S$ to this quotient lattice. There are several cases but all of them can be settled similarly. We only discuss those given by Figures~\ref{fighomopqr}, \ref{fighomopqrsupup}, and \ref{fighomogrc}. By \init{G.~}Gr\"atzer~\cite{ggtechnicallemma}, 
each of these cases would be quite easy, although a bit tedious. However, to indicate that the rest of cases are similar,  we give slightly more sophisticated arguments for them.
Note that these figures also use the injectivity of $f$; for example, this is why $p'\neq s'$ and $q'\neq r'$ in Figure~\ref{fighomopqrsupup}.

\begin{figure}[htc]
\centerline
{\includegraphics[scale=1.0]{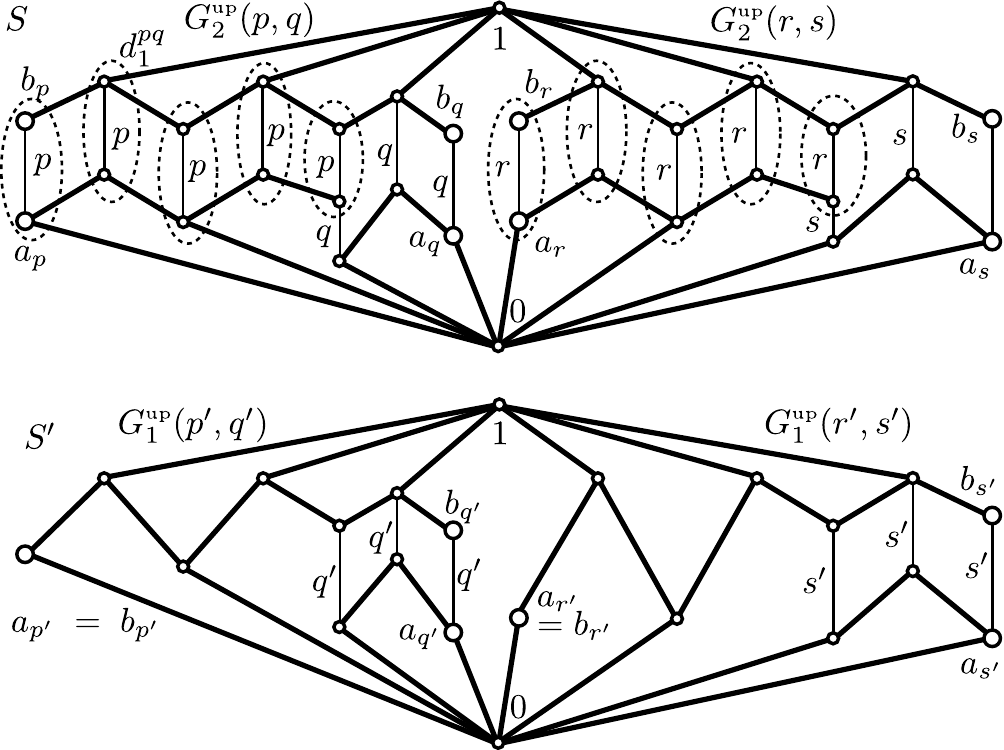}}
\caption{$\tuple{\text{up}, \text{up}}$, $\set{p,q,r,s}|=4$, and $\tuple{j,j',k,k'}=\tuple{2,1,2,1}$ \label{fighomopqrsupup}} 
\end{figure}

In case of Figure~\ref{fighomopqr}, let $H=\set{0,p,q,r,s,1}$ and \[
\nu=\bquos{\set{\pair p q, \pair q r, \pair r q, \pair r s}\cup (\set{0}\times H)\cup (H\times\set 1)}\text.
\]
(In general, the quasiordered set
$\tuple{H;\nu}$ is quite different from  $\tuple{H_1;\nu_1}$ and $\tuple{H_2;\nu_2}$.)
 Using that $S$ is a sublattice of the quasi-colored lattice $\alg L(H,\nu)$, see Lemma~\ref{keylemma} and  Definition~\ref{defLofquord}, it is easy to see that $\balpha_1\cup\balpha_2$ is a congruence of $S$. Namely, we can quite easily show that $\balpha_1\cup\balpha_2=\congen S{a_p,b_p}$. Clearly, $\congen S{a_p,b_p}$ collapses the $p$-colored edges. If it collapsed a $t$-colored edge for some $t\in\set{q,r,s,1}$ in $S$, then it would collapse the same edge (with the same color) in $
L(H,\nu)$, but then (C\ref{labqb}) would give $t\leq p$, a contradiction. 

In case of Figure~\ref{fighomopqrsupup}, let $\tuple{H;\nu}$ be the six element lattice in which there are exactly two maximal chains,  $\set{0\prec p\prec q\prec 1}$ and $\set{0\prec r\prec s\prec 1}$. The same argument as above shows that $\congen S{a_p,b_p}$ collapses the $p$-colored edges and only those, while $\congen S{a_r,b_r}$ collapses exactly the $r$-colored edges. To see that $\balpha_1\cup\balpha_2$ is a congruence, it suffices to show that $\balpha_1\cup\balpha_2=\congen S{a_p,b_p}\vee\congen S{a_r,b_r}$. Clearly, $\balpha_1\cup\balpha_2\subseteq \congen S{a_p,b_p}\vee\congen S{a_r,b_r}$. Assume that $\pair xy\in \covpairs S$ such that $\pair xy\in \congen S{a_p,b_p}\vee\congen S{a_r,b_r}$. In other words, $\congen S{x,y} \leq \congen S{a_p,b_p}\vee\congen S{a_r,b_r}$. Since a covering pair of a lattice always generates a join-irreducible congruence and the congruence lattice of a lattice is distributive, it follows that  
$\congen S{x,y} \leq \congen S{a_p,b_p}$ or $\congen S{x,y} \leq \congen S{a_r,b_r}$. Hence, $\pair xy\in \balpha_1$ or $\pair xy\in \balpha_2$, and we obtain the required inclusion, $\balpha_1\cup\balpha_2\supseteq \congen S{a_p,b_p}\vee\congen S{a_r,b_r}$.

For Figure~\ref{fighomogrc}, we use the same $\tuple{H;\nu}$ as for Figure~\ref{fighomopqr}, and practically the same argument shows that $\balpha_1\cup \balpha_2=\congen S{a_r,b_r}$.
\end{proof}

\begin{figure}[htc]
\centerline
{\includegraphics[scale=1.0]{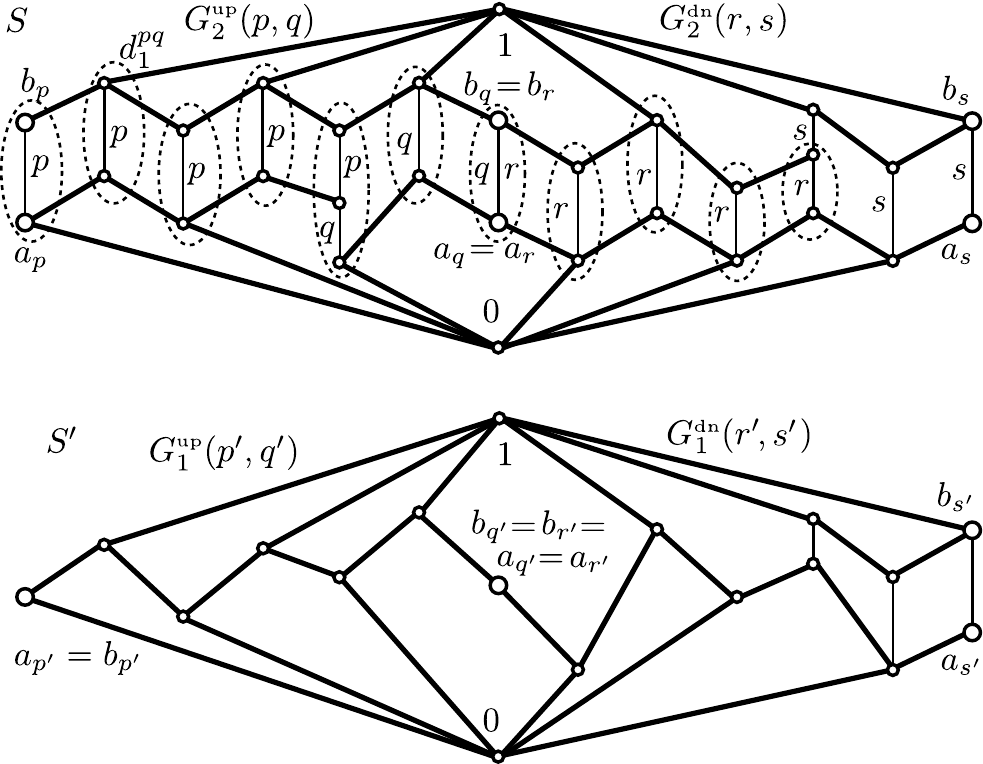}}
\caption{$\tuple{\text{up}, \text{dn}}$, $q=r$, and $\tuple{j,j',k,k'}=\tuple{2,0,2,1}$ \label{fighomogrc}} 
\end{figure}

\section{Completing the lattice theoretical part}\label{sectlatcompl}
For a quasiordered set $\tuple{H,\nu}$, we let $\Theta_\nu=\nu\cap\nu^{-1}$. It is known that $\Theta_\nu$  is an equivalence relation, and the definition 
\begin{equation}
\pair{\blokk x{\Theta_\nu}} {\blokk y{\Theta_\nu}} \in \nu/\Theta_\nu  \overset{\textup{def}}\iff \pair x y\in \nu
\end{equation}
turns the quotient set $H/\Theta_\nu$ into an ordered set $\tuple{H;\nu}/\Theta_\nu$, which is also denoted by $\tuple{H/\Theta_\nu;\nu/\Theta_\nu}$.
The following lemma is a straightforward consequence of (C\ref{labqa}) and (C\ref{labqb}), see \cite[Lemma 2.1]{czgprincc}, \cite[Lemma 3.1]{czgsingleinjectiveprinc}, or  \cite[Lemma \xilemma]{czginjlatcat}, where the inverse isomorphism is considered. Although the lemma is only formulated for the particular  quasi-colored lattices constructed in these papers, its easy proof makes it valid for every quasi-colored lattice, so it is time to formulate it more generally.

\begin{lemma}\label{lemmaQCOLisoPrinc}
For every quasi-colored lattice $\tuple{L,\leq; \gamma; H, \nu}$, $\princ L$ is isomorphic to $\tuple{H;\nu}/\Theta_\nu$ and the map $\tuple{\princ L; \subseteq} \to \tuple{H;\nu}/\Theta_\nu$, defined by $\cg xy\mapsto
\gamma(\pair xy)/\Theta_\nu$, is an order isomorphism.
\end{lemma}

As a consequence of this lemma and our construction,  or (the proof of) \cite[Lemma \xilemma]{czginjlatcat}, we obtain the following corollary.

\begin{corollary}\label{corolxilemma}
If $\tuple{H;\nu}$ is a quasiordered set satisfying \eqref{eqrHnuassum}, then 
the map
\[
\zeta_{H,\nu}\colon \tuple{H;\nu}/\Theta_\nu\to \tuple{\princ {\Nla(H,\nu)};\subseteq}
\]
defined by $\blokk p{\Theta_\nu} \mapsto\cg  {a_p}{b_p}$
is an order isomorphism.
\end{corollary}

\begin{proof}[Proof of Theorem~\ref{thmlat}]
Let $\posetfunct\colon \cat A\to \bposets$ be a faithful functor  as in the theorem, and let 
\[
\cat B:=\posetfunct(\cat A)\text.
\] 
For $X\in \obj{\cat A}$ and $f\in\mor{\cat A}$,  $\posetfunct(X)$ is an ordered set and $\posetfunct(f)$ is a monotone map; we will use the notation 
\[
\tuple{\psf X; \leq_X} := \posetfunct(X)
\,\text{ and }\, \psf f := \posetfunct(f) 
\text.
\]
In $\cat B $, two ordered sets with the same underlying set but different orderings are two distinct objects.   Since we do not want to identify distinct objects when we forget their orderings, we index the underlying sets as follows. For $\tuple{Y;\nu}\in \obj{\cat B }$, we let $\fofunct(\tuple{Y;\nu}):=Y\times\set{\nu}$. 
For $g\in \fmor(\tuple{Y_1;\nu_1}, \tuple{Y_2;\nu_2})$, we let
\[
g'=\fofunct(g)\colon Y_1\times \set{\nu_1}\to Y_2\times \set{\nu_2},\text{ defined by }\pair x{\nu_1}\mapsto \pair{g(x)}{\nu_2}\text.
\]
In this way, we have defined a  totally faithful functor $\fofunct\colon \cat B  \to \secatt$; the subscript comes from ``forgetful''. 
For $\tuple{X;\nu}$ and $x\in X$, if  $\nu$ is understood, we often write $X'$ and  $x'$ instead of $X\times\set{\nu}$ and   $\pair x\nu$. With this abbreviation, 
\begin{equation}
g'=\fofunct(g)\colon Y_1'\to Y_2'\text{ is defined by }x'\mapsto (g(x))'\text.
\label{eqrgprimedF}
\end{equation}
That is, for $X\in\cat A$,  $f\in\mor{\cat A}$, and $y\in \psf X$, 
\[\psv X = \fofunct(\posetfunct(X)) ,\,\text{ }\psv f=\fofunct(\posetfunct(f)),\,\text{ and }\,\psv f(y')=(\psf f(y))'  \text.
\]
The image 
\[
\cat C:=\fofunct(\cat B ) =(\fofunct\circ \posetfunct)(\cat A)
\] is a small concrete category, a subcategory of $\secatt$; its objects and morphisms are the $\psv X$ for $X\in \obj{\cat A}$ and the $\psv f$ for $f\in\mor{\cat A}$, respectively.
We claim that 
\begin{equation}
\text{all morphisms of }\cat C\text{ are monomorphisms.}
\label{eqrallmorinDmono}
\end{equation}
Since $\posetfunct$ is assumed to be faithful and $\fofunct$ is obviously faithful, \eqref{eqrallmorinDmono} will follow from the following trivial observation. 
\begin{equation}
\parbox{7.6cm}{If $F\colon \cat U\to\cat V$ is a faithful function, $\cat V=F(\cat U)$, and $f_1\in\mor{\cat U}$ is a monomorphism, then $F(f_1)$ is a monomorphism in $\cat V$.}
\label{eqrmonoffmono}
\end{equation}
To show this, assume that  $f_1\in \fmor_{\cat U}(X,Y)$ is a monomorphism and $f_2,f_3\in \fmor_{\cat U}(Z,X)$ such that $F(f_1)\circ F(f_2)=F(f_1)\circ F(f_3)$. Then
$F(f_1\circ f_2)=F(f_1\circ f_3)$, which implies $f_1\circ f_2=f_1\circ f_3$ since $F$ is faithful. Hence, $f_2=f_3$ and $F(f_2)=F(f_3)$. This proves \eqref{eqrmonoffmono} and, consequently, \eqref{eqrallmorinDmono}.

Although $\psv X=  \fofunct(\tuple{\psf X;\leq_X})=\fofunct(\posetfunct(X))$ is
only a set for $X\in\obj{\cat A}$, we shell use the ordering $\leq_X'$ induced by $\leq_X$ on it as follows: for $x,y\in\psf X$, 
\begin{equation}
\pair x\nu\leq_X'\pair y\nu 
\overset{\textup{def}}\iff x\leq_X y\text{, that is, }x'\leq_X' y' \overset{\textup{def}}\iff x\leq_X y
\text.
\label{eqrleqXprime}
\end{equation}
As a consequence of  \eqref{eqrleqXprime}, we have that
\begin{equation}
 0' \text{ resp.\ }1'\text{ are is the least resp.\  greatest element of }\tuple{\psv X,\leq_X'}\text. 
\label{eqrzvevR}
\end{equation}

Next, we let  
\[
{\cat D}:=\celfunct(\cat C)\text.
\]
By \eqref{eqrallmorinDmono} and Theorem~\ref{thmcat}, 
\begin{equation}
\text{all morphisms of } {\cat D} \text{ are injective maps.}
\label{eqrInCelallinj}
\end{equation}
(This is why we can apply  Lemma~\ref{lemmahomomain} soon.)
Since we have three functors already, it is worth defining their composite,
\begin{equation*}
\compfunct:= \celfunct\circ \fofunct \circ \posetfunct,\quad\text{from } \cat A\text{ to } {\cat D}\text.
\end{equation*}

Next, for $X\in \obj{\cat A}$, we define a relation $\nu_X$ on the set $\compfunct(X)=\celfunct(\psv X)$ as follows:  for eligible triplets $c_1, c_2\in \compfunct(X)=\celfunct(\psv X)$, 
\begin{equation}
\pair{c_1}{c_2}\in \nu_X \overset{\textup{def}}\iff \xpceltransf{\psv X}{c_1} \leq_X' \xpceltransf{\psv X}{c_1}  \text.
\label{eqrefnuprimedef}
\end{equation}
Clearly, $\nu_X'$ is a quasiorder.
The set of least elements of $\tuple{\compfunct(X);\nu_X}$ will be denoted by $Z(\compfunct(X))$. Similarly, $U(\compfunct(X))$ will stand for the set of largest elements. \eqref{eqrzvevR} and   \eqref{eqrefnuprimedef} make it clear that 
\begin{equation}
\begin{aligned}
Z(\compfunct(X))&=\set{c\in \compfunct(X): \xpceltransf{\psv X}{c}=0'}\text{, and} 
\cr
U(\compfunct(X))&=\set{c\in \compfunct(X): \xpceltransf{\psv X}{c}=1'}\text.\cr
\end{aligned}
\label{eqrefZandU}
\end{equation}
It also follows from \eqref{eqrefnuprimedef} that these sets are nonempty, because
\begin{equation*}
\begin{aligned}
&\trivcom(0')=\tuple{1_{\psv X},0',0'}\in Z(\compfunct(X))\text{, and}
\cr
&\trivcom(1')=\tuple{1_{\psv X},1',1'}\in U(\compfunct(X))\text.
\end{aligned}
\end{equation*}
Since $0\neq 1$ in $\posetfunct(X)$,  the distinguished eligible triplets  
\begin{equation}
\trivcom(0')\in  Z(\compfunct(X))\text{ and }\trivcom(1')\in  U(\compfunct(X))\text{ are distinct.}
\label{eqrdcmtDistinct}
\end{equation}
Hence, for $X\in \obj{\cat A}$, Definition~\ref{defLofquord} allows us to consider the quasi-colored lattice
\begin{equation}
\begin{aligned}
\alg \Nla(&\compfunct(X),\nu_X)=\cr
&\tuple{ \Nla(\compfunct(X),\nu_X), \lambda_{\compfunct(X),\nu_X}; \gamma_{\compfunct(X),\nu_X}; \compfunct(X),\nu_X}\text.  
\end{aligned}
\label{eqrefKHnudef}
\end{equation}
For $f\in\mor{\cat A}$, $\psv f=(\fofunct\circ \posetfunct)(f)$ and 
$\compfunct(f)$ are only maps  between two sets. However, \eqref{eqrleqXprime} and \eqref{eqrefnuprimedef}, respectively, allow us to guess that these maps are monotone; these properties are conveniently formulated in the form  \eqref{eqrfgfofpomonontone} below and \eqref{erismonotmap} later. 
First, we claim that  for $X,Y\in\obj{\cat A}$ and $f\in\fmor_{\cat A}(X,Y)$, 
\begin{equation} 
\psv f= (\fofunct\circ\posetfunct)(f)
\colon\tuple{\psv X;\leq_X'}\to \tuple{\psv Y;\leq_Y'}
\text{ is a monotone map.}
\label{eqrfgfofpomonontone}
\end{equation}
To show this, assume that $x_1,x_2\in \psf X$ such that $x_1'\leq_X' x_2'$. By \eqref{eqrleqXprime}, $x_1\leq_X x_2$. Since $\psf f=\posetfunct(f)$ is a monotone map, $\psf f(x_1)\leq_Y \psf f(x_2)$. Hence, by \eqref{eqrleqXprime}, we have that $(\psf f(x_1))'\leq_Y' (\psf f(x_2))'$. Thus, applying \eqref{eqrgprimedF} for $\psf f$ and $x_i$, 
\begin{equation*}
\psv f (x_1')
=(\psf f(x_1))'  
  \leq_Y' 
(\psf f(x_2))' = 
\psv f(x_2'),
\end{equation*}
which proves \eqref{eqrfgfofpomonontone}.

Second, we are going to show that for $X,Y\in\obj{\cat A}$ and $f\in\fmor_{\cat A}(X,Y)$,
\begin{equation}
\begin{aligned}
\compfunct(f)\colon \tuple{\compfunct(X);\nu_X}\to \tuple{\compfunct(Y);\nu_Y} 
\text{ is a monotone map.}
\end{aligned}
\label{erismonotmap}
\end{equation}
So let $X,Y\in\obj{\cat A}$ and $f\in\fmor_{\cat A}(X,Y)$. 
Since $\celtransf$ is a natural transformation by Theorem~\ref{thmcat},
the diagram
\begin{equation*}
\begin{CD}
\compfunct(X)=\celfunct(\psv X)  @>{\compfunct(f)}>> \compfunct(Y)=\celfunct(\psv Y) \\
@V{\xceltransf{\psv X}}VV   @V{\xceltransf{\psv Y}}VV   \\
\psv X @>{\psv f}>> \psv Y
\end{CD}
\end{equation*}
commutes. That is, for every $c\in \compfunct(X)$, 
\begin{equation}
\xpceltransf{\psv Y}{\compfunct( f)(c)}= \psv f(\xpceltransf{\psv X}{c})\text.
\label{eqrHoComitmaps}
\end{equation}
Assume that $\pair{c_1}{c_2}\in \nu_X$.  By \eqref{eqrefnuprimedef}, $\xpceltransf{\psv X}{c_1}\leq_X'\xpceltransf{\psv X}{c_2}$. By \eqref{eqrfgfofpomonontone}, this gives that 
$\psv f(\xpceltransf{\psv X}{c_1}) \leq_Y' \psv f(\xpceltransf{\psv X}{c_2})$. Combining this inequality with \eqref{eqrefnuprimedef} and \eqref{eqrHoComitmaps}, we obtain that $\pair{\compfunct(f)(c_1)}{\compfunct(f)(c_2)}\in \nu_Y$, proving \eqref{erismonotmap}.

Our next task is to show that 
\begin{equation}
\begin{aligned}
\compfunct(f)(Z(\compfunct(X)))&\subseteq Z(\compfunct(Y))\text{ and }\cr
\compfunct(f)(U(\compfunct(X)))&\subseteq U(\compfunct(Y))\text.
\end{aligned}
\label{eqrcompfunctpresZU}
\end{equation}
Assume that $c\in Z(\compfunct(X))$.  By \eqref{eqrefZandU}, $\xpceltransf{\psv X}{c}=0'$. Since $\psf f= \posetfunct(f)\in \mor{\cat B}\subseteq \mor{\bposets}$, $\psf f$ is $0$-preserving. Hence, by  \eqref{eqrgprimedF} and \eqref{eqrHoComitmaps}, 
\[\xpceltransf{\psv Y}{\compfunct(f)(c)} =  \psv f(\xpceltransf{\psv X}{c})=\psv f(0') = (\psf f(0))' = 0'\text.
\]
By \eqref{eqrefZandU}, this means that $\compfunct(f)(c_3)\in Z(\compfunct(Y))$.  This proves 
the first half of \eqref{eqrcompfunctpresZU}; the second half follows in the same way.

Now,  we are in the position to define a functor $\liftfunct \colon \cat A\to \latasdf$ as follows. For $X\in \obj{\cat A}$ and $f\in \fmor_{\cat A}(X,Y)\subseteq \mor{\cat A}$, we let
\begin{equation}
\begin{aligned}
\liftfunct(X)&:= \Nla(\compfunct(X);\nu_X),\text{ see \eqref{eqrefKHnudef},}\cr
\liftfunct(f)&:= \text{the unique $\set{0,1}$-preserving lattice homomorphism}\cr
&\kern 17pt \text{that Lemma~\ref{lemmahomomain} associates with }\compfunct(f);
\end{aligned}
\label{eqrliftFuDef}
\end{equation}
it follows from \eqref{eqrInCelallinj}, \eqref{eqrdcmtDistinct},  \eqref{erismonotmap}, and \eqref{eqrcompfunctpresZU} that Lemma~\ref{lemmahomomain} is applicable. We are going to show that  $\liftfunct$ is a functor from $\cat A$ to $\latasdf$.
By Lemma~\ref{keylemma}, Definition~\ref{defLofquord}, and   \eqref{eqrefKHnudef}, we have that $\liftfunct(X)\in \obj{\latasdf}$. 
By Lemma~\ref{lemmahomomain}, $\compfunct(f)\in\mor{\latasdf}$. If $f=1_X\in \fmor_{\cat A}(X,X)$, then $\compfunct(f)$ is the identity map since $\compfunct$ is a functor, and it follows from \eqref{eqrefgprops} and the uniqueness part of Lemma~\ref{lemmahomomain} that $\liftfunct(f)$ is the identity map $1_{\liftfunct(X)}$. Finally, assume that $X,Y,Z\in \obj{\cat A}$, $f_1\in\fmor_{\cat A}(X,Y)$, and $f_2\in\fmor_{\cat A}(Y,Z)$.   
We have to show that $\liftfunct(f_1\circ f_2)=\liftfunct(f_1)\circ \liftfunct(f_2)$.  By \eqref{eqrefgprops} and the uniqueness part of Lemma~\ref{lemmahomomain}, it suffices to show that 
\begin{equation}
\liftfunct(f_1\circ f_2)(a_p)=(\liftfunct(f_1)\circ\liftfunct(f_2))(a_p)
\label{eqrwhatapshould}
\end{equation}
for all eligible triplets $p\in \compfunct(X)$, and similarly for $b_p$. 
By \eqref{eqrefgprops} and  \eqref{eqrliftFuDef}, we have the following rule of computation:
\begin{equation}
\liftfunct(f)(a_p)=a_{\compfunct(f)(p)}\text.
\label{eqrrulecompap}
\end{equation}
We know that $\compfunct$, as composite of three functors, is a functor. Therefore, $\compfunct(f_1\circ f_2)= \compfunct(f_1)\circ  \compfunct(f_2)$. Using this equality and \eqref{eqrrulecompap}, we have 
\begin{align*}
\liftfunct(f_1\circ f_2)(a_p)&=a_{\compfunct(f_1\circ f_2)(p)} =
a_{(\compfunct(f_1)\circ \compfunct(f_2))(p)} \cr
&=a_{\compfunct(f_1)( \compfunct(f_2))(p))} =\liftfunct(f_1)( a_{\compfunct(f_2)(p)})\cr
&= \liftfunct(f_1)(\liftfunct(f_2)(a_p)) =
(\liftfunct(f_1)\circ \liftfunct(f_2))(a_p)
\text.
\end{align*}
Thus, \eqref{eqrwhatapshould} holds, and $\liftfunct\colon\cat A\to\latasdf$ is a functor, as required.

Clearly, the composite of faithful or totally faithful functors is a  faithful or totally faithful functor, respectively.  By Theorem~\ref{thmcat}, $\celfunct$ is totally faithful. 
So is $\fofunct$. Therefore, 
$\compfunct=\celfunct\circ \fofunct \circ \posetfunct$ is faithful, and it is totally faithful if so is  $\posetfunct$. Hence, it follows from \eqref{eqrrulecompap} that $\liftfunct$ is faithful. Furthermore, if $\posetfunct$ is totally faithful and $X\neq Y\in\obj{\cat A}$, then the same property of $\compfunct$ gives that $\set{a_p: p\in \compfunct(X)}$
is distinct from $\set{a_p: p\in \compfunct(Y)}$. Hence, it follows from Remark~\ref{remstrongfaithfconstr} and  
\eqref{eqrliftFuDef} that $\liftfunct(X)\neq \liftfunct(Y)$. Consequently, $\liftfunct$ is totally faithful if so is $\posetfunct$. 

Finally, we are going to prove that $\liftfunct$ lifts $\posetfunct$ with respect to $\fprinc$. The isomorphism provided by Corollary~\ref{corolxilemma} will be denoted by $\zeta_X$. That is, 
\begin{equation}
\begin{aligned}
\zeta_X\colon \tuple{\compfunct(X);\nu_X} /\Theta_{\nu_X}\to{} &\tuple{\princ{\Nla(\compfunct(X),\nu_X)};\subseteq}\cr
&\overset{\textup{\eqref{eqrliftFuDef}}}=(\fprinc \circ \liftfunct)(X),\cr
&\kern-100 pt \text{defined by }q/\Theta_{\nu_X}\mapsto \cg {a_q}{b_q},
\label{eqrzetaXdef}
\end{aligned}
\end{equation}
is an order isomorphism. 
The map  $\xceltransf X\colon\tuple{\compfunct(X);\nu_X}\to \tuple{\psv X;\leq_X'}$ from Definition~\ref{defcomet} is monotone by \eqref{eqrefnuprimedef}. This map is surjective, because   
$\xpceltransf{\psv X}{\trivcom(p')}=p'$ holds for every $p\in \psf X$, that is, $p'\in \psv X$. Furthermore, if $p'\leq_X' q'$, then $\pair{\trivcom(p')}{\trivcom(q')}\in \nu_X$ by \eqref{eqrefnuprimedef}, which means that the ordering $\leq_X'$ equals 
the $\xceltransf X$-image of $\nu_X$.
Hence, using a well-known fact about orders induced by quasiorders,  the map
\[
\tuple{\compfunct(X);\nu_X} /\Theta_{\nu_X}\to \tuple{\psv X;\leq_X'},\,\text{ defined by  }\, q/\Theta_{\nu_X}\mapsto  \xpceltransf{\psv X} q,
\]
is an order isomorphism. So is its inverse map,
\[\tuple{\psv X;\leq_X'}\to  \tuple{\compfunct(X);\nu_X} /\Theta_{\nu_X}\,\text{ defined by  }\, p'\mapsto \trivcom(p')/\Theta_{\nu_X}\text.
\]
Since $\tuple{\psf X;\leq_X}\to \tuple{\psv X;\leq_X'}$, defined by $x \mapsto x'$, is also an isomorphism by \eqref{eqrleqXprime}, the composite 
\begin{equation}
\xi_X\colon \tuple{\psf X;\leq_X}\to \tuple{\compfunct(X);\nu_X} /\Theta_{\nu_X},
\text{ defined by }p\mapsto \trivcom(p')/\Theta_{\nu_X},
\label{eqrxiXdef}
\end{equation}
of the two isomorphisms is also an order isomorphism. So we can let\begin{equation}
\xtautr X:= \zeta_X\circ\xi_X,
\text { which is an order isomorphism}
\label{eqrcelltrisiso}
\end{equation}
from $\posetfunct(X)=\tuple{\psf X;\leq_X}$ to $(\fprinc\circ\liftfunct)(X)$ by  \eqref{eqrzetaXdef} and \eqref{eqrxiXdef}. As the last part of the proof, we are going to show that
 $\tautr\colon\posetfunct\to \fprinc\circ\liftfunct$ is a natural isomorphism. By \eqref{eqrcelltrisiso}, we only have to show that it is a natural transformation. To do so,  assume that $X,Y\in\obj{\cat A}$ and $f\in\fmor_{\cat A}(X,Y)$.
Besides $\psf f=\posetfunct(f)$ and $\psv f=(\fofunct\circ \posetfunct)(f)$,
we will use the notation  $h:=(\fprinc\circ\liftfunct)(f)$. 
We have to show that the diagram
\begin{equation}
\begin{CD}
\posetfunct(X)=\tuple{\psf X;\leq_X}  @>{\bdiakern\cordiakern\psf f=\posetfunct(f)\cordiakern\bdiakern}>> \posetfunct(Y)=\tuple{\psf X;\leq_Y}  \\
@V{\xtautr X }VV  @V{\xtautr Y   }VV   \\
 (\fprinc\circ\liftfunct)(X) @>{\bdiakern h=(\fprinc\circ\liftfunct)(f)\bdiakern}>> (\fprinc\circ\liftfunct)(Y)
\end{CD}
\label{CDliftnattrf}
\end{equation}
commutes.
First, we investigate the map $h$. For a triplet $q\in \compfunct(X)$, we have that $\liftfunct(f)(a_q)= a_{\compfunct(f)(q)}$ by \eqref{eqrrulecompap}. Analogously, $\liftfunct(f)(b_q)= b_{\compfunct(f)(q)}$. Therefore, applying the definition of $\fprinc$ for the $\set{0,1}$-lattice homomorphism $\liftfunct(f)\colon \liftfunct(X)\to \liftfunct(Y)$, see \eqref{eqrefsljGfB} and \eqref{eqrfuncPrincDf}, 
we have that
\begin{equation}
h(\cg{a_q}{b_q})= \cg{a_{\compfunct(f)(q)}}
{b_{\compfunct(f)(q)}}\text.
\label{eqrhofcmputE}
\end{equation}
Consider an arbitrary $p\in \posetfunct(X)$. By \eqref{eqrzetaXdef}, \eqref{eqrxiXdef}, and \eqref{eqrcelltrisiso}, 
\begin{equation}
\xtautr X(p)= \zeta_X(\xi_X(p))
)= \zeta_X(\trivcom(p')/\Theta_{\nu_X}) = \cg{a_{\trivcom(p')}}{b_{\trivcom(p')}}\text.
\label{eqrHowcltracTs}
\end{equation}
Hence, \eqref{eqrhofcmputE} yields that
\begin{equation}
(h\circ\xtautr X)(p) =  \cg{a_{\compfunct(f)(\trivcom(p'))}}
{b_{\compfunct(f)(\trivcom(p'))}}\text.
\label{eqrcmprghTdn}
\end{equation}
On the other hand, using \eqref{eqrHowcltracTs} for $Y$ and $\psf f(p)$ instead of $X$ and $p$, 
\begin{equation}
(\xtautr Y\circ \psf f)(p) =  \xtautr Y(\psf f(p))= \cg{a_{\trivcom(\psf f(p)')}}{b_{\trivcom(\psf f (p)')}}\text.
\label{eqrdnrgTcomp}
\end{equation}
We are going to verify that  \eqref{eqrcmprghTdn} and \eqref{eqrdnrgTcomp} give the same principal congruences.
 Motivated by (C\ref{labqa}), we focus on the colors of the  respective  ordered pairs that generate these two principal congruences. By the construction of our quasi-colored lattices, see Figure~\ref{fig-m33}  and \eqref{eqrgammahzuIJ}, these colors are 
$c_1:=\compfunct(f)(\trivcom(p'))$, in \eqref{eqrcmprghTdn},  and  $c_2:=\trivcom(\psf f(p)')$, in \eqref{eqrdnrgTcomp}. By \eqref{eqrtrrtrainv},   \eqref{eqrgprimedF},  and   \eqref{eqrHoComitmaps}, 
\begin{align*}
\xpceltransf{\psv Y}{c_1} &= \xpceltransf{\psv Y}{\compfunct(f)(\trivcom(p'))} 
 = \psv f(\xpceltransf{\psv X}{\trivcom(p')}) \cr
& =\psv f(p') = \psf f (p)'
= \xpceltransf{\psv Y}{\trivcom(\psf f(p)')} = 
\xpceltransf{\psv Y}{c_2}\text.
\end{align*}
Hence,   \eqref{eqrefnuprimedef} yields that  $\pair{c_1}{c_2}\in \nu_Y$ and $\pair{c_2}{c_1}\in \nu_Y$. Thus, we conclude from (C\ref{labqa}) that \eqref{eqrcmprghTdn} and \eqref{eqrdnrgTcomp} are the same principal congruences, which means that the diagram given in \eqref{CDliftnattrf} commutes. This proves that $\liftfunct$ lifts $\posetfunct$ with respect to $\fprinc$, as required.  
\end{proof}

\end{document}